\documentclass[reqno]{amsart}
\usepackage[utf8]{inputenc}
\usepackage{enumitem}

\usepackage{amsmath,amssymb,amsbsy,amsfonts,amsthm,latexsym,amsopn,amstext,mathtools,leftidx,
            amsxtra,euscript,amscd,verbatim,epsf,colordvi,graphics}
\usepackage{ytableau}
\usepackage{color}
\usepackage{pgf,tikz}
\usetikzlibrary{decorations.pathreplacing}
\usetikzlibrary{arrows}
\usepackage{mathrsfs}
\usepackage{graphicx}

\usepackage{hyperref}

\usepackage{cleveref}
\usepackage[margin=11pt,font=footnotesize,labelfont=bf]{caption}
\newtheoremstyle{tplain}{3pt}{3pt}{\rmfamily}{}{\bfseries}{.}{0.5em}{}
\theoremstyle{tplain}

\usepackage{pgf,tikz}
\usepackage{tcolorbox}
\usepackage[enableskew]{youngtab}
\definecolor{darkgreen}{cmyk}{1,0,1,0}
\newtheorem{thm}{Theorem}
\newtheorem{lem}{Lemma}
\newtheorem{ex}{Example}
\newtheorem{cor}{Corollary}
\newtheorem{prop}{Proposition}
\newtheorem{obs}{Remark}
\newtheorem{defi}{Definition}
\def \cal{\mathcal}
\def \rm {\mathrm}
\def \mbf {\mathbf}

\def\rev{{\rm rev}}
\def\Z { \mathbb{Z}}
\def\C { \mathbb{C}}
\def\LRAII {\mathsf{LR}}
\def \r {\mathbf{r}}
\def \t {\mathbf{t}}
\def \redu {\mathrm{red}}
\def\c { \mathrm{c}}
\def \k {\mathfrak{k}}
\def\wt { \mathrm{wt}}
\def\shape { \mathrm{shape}}
\def\rem { \mathrm{rem}}
\def\suc { \mathrm{suc}}
\def \s {\mathbf{s}}
\def \a {\mathbf{a}}
\def \A {\mathbf{A}}
\def \b {\mathbf{b}}
\def \c {\mathbf{c}}
\def \d {\mathbf{d}}
\def \x {\mathbf{x}}
\def \SST {SST_{2n}}
\def \SpT {SpT_{2n}}

\newcommand{\oo}{\color{blue}}
\newcommand{\blue}{\color{blue}}
\newcommand{\red}{\color{red}}
\newcommand{\green}{\color{green}}
\newcommand{\bro}{\color{brown}}

\newcommand{\ora}{\color{orange}}

\makeatletter
\newcommand*\bigcdot{\mathpalette\bigcdot@{.5}}
\newcommand*\bigcdot@[2]{\mathbin{\vcenter{\hbox{\scalebox{#2}{$\m@th#1\bullet$}}}}}
\makeatother

\newcommand*\circled[1]{\tikz[baseline=(char.base)]{
            \node[shape=circle,draw,inner sep=1pt] (char) {#1};}}

\newcommand{\YT}[3]{
\vcenter{\hbox{
\begin{tikzpicture}[x={(0in,-#1)},y={(#1,0in)}] 
\foreach \rowi [count=\i] in {#3} {
 \foreach \e [count=\j] in \rowi {
  \draw (\i,\j) rectangle +(-1,-1);
  \draw (\i-0.5,\j-0.5) node {$#2\e$};
 }
}
\end{tikzpicture}
}}
}

\title[Inverse reduction map] {
The inverse  reduction map of a symplectic column by decreasing the rank by one }

\author{Olga Azenhas}
\address{ University of Coimbra, CMUC, Department of Mathematics, Portugal}
\email{oazenhas@mat.uc.pt}



\begin{document}

\begin{abstract} 
We have previously given a factorization of a symplectic column under the action of the parity involution which enabled to
 explicitly have
written the inverse of the reduction map  in the quantum Littlewood-Richardson bijection.
Watanabe has written the reduction map as  a composition of several maps, among them, combinatorial $R$-matrices on single columns and a reduction map  on a shorter  column with rank reduced by one.
We now use this approach to  write the inverse of the reduction map on a symplectic column by detecting the corresponding symplectic column of rank reduced by one and thus avoiding going through several map compositions.
\end{abstract}
\maketitle

\tableofcontents

\section{Introduction}
Given   $\lambda \in Par_{\le 2n}$  a partition with at most $2n$ parts,
the quantum Littlewood-Richardson (LR) map $LR^{AII}$ \cite[Theorem 3.1.4]{watanabe}
is an one-to-one
assignment of a semi-standard tableau $T$ of shape $\lambda$  to a pair $(P^{II}(T), Q^{II}(T)) $   consisting of a symplectic tableau in $ SpT_{2n}(\mu)$,  for some shape $ \mu\in Par_{\le n}$, a partition with at most $n$ parts, where $\mu\subseteq \lambda$, and   a recording tableau of skew shape $\lambda/\mu$ in $Rec_{2n}(\lambda/\mu)$,

\begin{align} \label{lrsII}
 SST_{2n}(\lambda) \overset{\sim}\longrightarrow
\bigsqcup_{\begin{smallmatrix}\mu\in Par_{\le n}\\
\mu\subseteq\lambda\\
Q\in Rec_{2n}(\lambda/\mu)\end{smallmatrix}}SpT_{2n}(\mu) \times \{Q\}.
\end{align}

The set of recording tableaux $Rec_{2n}(\lambda/\mu)$  is in natural bijection \cite{watanabe, azreco} with the set of LR-Sundaram tableaux  $LRS_{2n}(\lambda/\mu)$ \cite{sundaram},
\begin{align}\label{record}Rec_{2n}(\lambda/\mu) \overset{\sim}\longrightarrow LRS_{2n}(\lambda/\mu).
\end{align}

For $T\in  SST_{2n}(\lambda)$, the computation of the symplectic tableau $P^{II}(T)\in  SpT_{2n}(\mu)$ requires in the first step the action of the   reduction map, $\redu$, on its first column, say ${\bf a}\in SST_{2n}(\varpi_l)$ with $l$ the length of $\lambda$.
The action of  the reduction map, $\redu$,  to $\bf a$, returns a symplectic column $\redu({\bf a})\in SpT_{2n}(\varpi_t)$  for some
$t\in[0,n]$ such that $0\le t\le min\{l,2n-l\}$ and $l-t\in 2\Z$ \cite[Proposition 4.3.6.]{watanabe}. 
Next by column Schensted insertion \cite{fulton, stanley}, insert $\redu({\bf a})$ into $T'$, the tableau $T$ with the column $\a$ suppressed, $(T'\leftarrow \redu({\bf a}))$, to get the successor of $T$, $\suc(T)$,
\begin{align}T\mapsto  (\redu( {\bf a}), T')\mapsto \suc(T):=
\redu({\bf a})\otimes T'=(T'\leftarrow\redu({\bf a}) ).
\end{align}
If $\suc(T)$ is not a symplectic tableau it means that its first column is not symplectic and   the previous procedure is iterated. 
The  successor will have  strictly less   cells otherwise the successor stabilizes which means that its first column is symplectic. The termination of the procedure  in a finite number of steps is guaranteed by the property $\redu(\a)=\a$ if and only if $\a$ is a symplectic column.

 The  computation of the reduction map on a column $\a\in \SST(\varpi_l)$ 
is defined via the removal operation, $\rem$, on the column $\a$, \cite[Definition 2.5.1]{watanabe} together with their  properties \cite[Proposition 4.2.2, 4.2.7,  4.2.8]{watanabe} in the sense that $\redu({\bf a})={\bf a} \setminus \rem({\bf a})$.
The set $\rem({\bf a})$ of removal entries of $\a$ consists of certain  pairs of entries, an odd entry followed  with an even entry (see Proposition \ref{prop:rem}), to be removed from $\bf a$ so that the resulting column is symplectic.  The inverse of the reduction map consists in adding such pairs to a symplectic column so that they can be removed \cite{azreduction}

To prove the injectivity of the reduction map, Watanabe has written the reduction map as  composition of several maps \cite[Corollary 4.4.3]{watanabe}, notably among them $R$-combinatorial matrices on single columns \cite{rmatrix} and a reduction map acting on  columns of rank reduced by one (see Subsection \ref{sec:Rcombinatorial}):
For a fixed $n\in\mathbb{N}$, let $l\in[0,2n]$. The map

\begin{align}\redu=\redu_l:SST_{2n}(\varpi_l)&\rightarrow\bigsqcup_{\begin{smallmatrix} 0\le t\le min\{l,2n-l\}\\
l-t\in 2\mathbb{Z}
\end{smallmatrix}}SpT_{2n}(\varpi_t), \quad {\bf a}\mapsto \redu({\bf a})\nonumber\\
\label{redumapthm}
\end{align}
is injective and is inductively defined as follows
\begin{enumerate}
\item for $l=0,1$, either $t=0$ or respectively $t=1$, in which cases $\redu=id$,
    \item for $l= 2$, if ${\bf a}=(1,2)$, $\redu(\bf a)=()$ and $t=0$, otherwise $\redu=id$ and $t=2$,
    \item for $l\ge 3$,
\begin{align}\label{decomposereduthm}\redu=\pi\circ\bigwedge\circ(K^{-1},id)\circ R_{t-1,2n-1}\circ(\redu_{l-1},id)\circ R_{2n-1,l-1}\circ (K,id)\circ \bigvee, 
\end{align}
\end{enumerate} 
where $\redu_{l-1}$ is the reduction map acting on $\SST(\varpi_{l-1})$, $R_{t-1,2n-1}$ and $R_{2n-1,l-1}$ are $R$ combinatorial matrices acting on single columns, and  $ \bigvee,\bigwedge$ are operators acting on a column in $SST_{2n}(\varpi_l)$ splitting it into two parts so that the last entry is detached respectively adjoining an entry to the bottom of a column in $SpT_{2n}(\varpi_{t-1})$, and $K$ is an operator based on the parity involution $\s$, Definition \ref{s}, and the $\vee$-operator $.^\vee$ Section \ref{sec:final} (see also \cite[Sections 4.3,~4.4]{watanabe}).

Recall that the parity involution $\s$, Definition \ref{s}, swaps an even number with the previous odd number and an odd number with the next even number. In \cite[Theorem 1]{azreduction},
we have decomposed  a symplectic column by detaching  its  maximal nonempty intervals that are fixed by the parity involution $\s$, this  is the part of the part of the symplectic column that is fixed by that involution, while the remaining  part consists of the subword whose image under the parity involution is disjoint of the symplectic column. The decomposition is uniquely determined by the symplectic column and the symplectic conditions satisfied by those pieces encode the needed information to explicitly write the inverse of the reduction map  in the quantum Littlewood-Richardson bijection.
The surjectivity and the inverse of the inverse of reduction map were solved in this way in \cite[Theorem 7]{azreduction}.

Here we consider a different approach to the  reduction map $\redu$ \eqref{redumapthm} namely as a composition of the maps in \eqref{decomposereduthm} and prove in this framework that $\redu$ is surjective in  Proposition \ref{prop:surjective}. The proof is lengthy and detailed while providing additional insight. As a bonus we get in Theorem \ref{th:main} an explicit  presentation of the inverse reduction map on a symplectic column by detaching the corresponding symplectic column of rank reduced by one. This avoids the need to go through the composition of several maps, by detaching the corresponding symplectic column of rank reduced by one. In this presentation, the computation of the inverse reduction  map on a symplectic column essentially reduces to the computation of the inverse of the reduction map on a symplectic column of rank reduced by one.

\bigskip


\subsection{Organization}  The paper is organized into four sections. Sections \ref{sec:containmentmain} and \ref{sec:redu} provide the necessary definitions on symplectic tableaux and columns, removal operations to define the reduction map and the inverse reduction map.
Section \ref{sec:final}  is the main one where we establish the new results in Proposition \ref{prop:surjective} and in  Main Theorem \ref{th:main}. Last subsection illustrates the Main Theorem.
\section*{Acknowledgements}
The author acknowledges financial support by the Centre for Mathematics of the University of Coimbra (CMUC, https://doi.org/10.54499/UID/00324/2025) under the Portuguese Foundation for Science and Technology (FCT), Grants UID/00324/2025 and UID/PRR/00324/2025.

\section{Symplectic tableaux and symplectic columns}\label{sec:containmentmain}

 Throughout we fix the following notation. Given $a,b\in \Z$, $[a,b]:=\{x\in \Z: a\le x\le b\}$. When $a\in\Z^+$, we just write $[a]:=[1,a]$.
Given $l\ge 0$, $\varpi_l:=(\underbrace{1,\dots,1}_l)\in \Z^l$.
\subsection{Preliminaries}
A \emph{partition} $\lambda$ is a weakly decreasing sequence of nonnegative integers $\lambda_1\ge \lambda_2 \ge \cdots$
such that $\lambda_k =0$ for some $k \ge 1$. The   maximal $i$ such that $\lambda_i> 0$ is called the \emph{number}
of \emph{parts} or \emph{length} of $\lambda$, denoted $\ell(\lambda)$.
For each $m\ge 0$, the set of partitions of length at most $m$ is denoted by \emph{$Par_{\le m}$}. We assume the inclusion $Par_{\le m}\subseteq Par_{\le k}$ whenever $k\ge m$. Thus
We often write the partition $\lambda$ as a vector
 $\lambda = (\lambda_1, \lambda_2, \dots,\lambda_k)$ for $k \ge  \ell(\lambda)$. The empty partition is the empty sequence $()$ and is regarded as the unique partition of length zero.

 A partition $\lambda$ is identified
with its \emph{Young diagram} $D(\lambda)$ which is a left and top justified collection of boxes (or cells)
with $\lambda_k$ many boxes in the $k$th row for all $k \in \mathbb{Z}_> 0$. In particular, the empty Young diagram and  the  partition $()$ are identified. The number of cells of $D(\lambda)$ is  the sum of the parts of $\lambda$ and is denoted by $|\lambda|$.
The boxes or cells of the Young diagram of $\lambda$ are identified by its  coordinates $(i,j)$ in the matrix style, that is, $1\le i\le \ell(\lambda)$ and $1\le j\le \lambda_i$.

 A \emph{
tableau} $T$ of  shape $\lambda$ is a map (or a filling of $D(\lambda)$)
 $$T:D(\lambda)\rightarrow \mathbb{Z}_>0, \; (i,j)\mapsto T(i,j),$$
assigning a positive integer to each box of $D(\lambda)$. 
We say that the tableau $T$ is \emph{semi-standard} if in addition the assignment is
such that it is weakly increasing as we go from left to right along a row and strictly
increasing as we go from top to bottom along a column, 
\begin{align}&T(i,j)\le T(i,j+1),\; T(i,j)<T(i+1,j), \mbox{ for all $(i,j)\in D(\lambda)$,}\nonumber
\end{align}
where we set $T(a,b):=\infty$ if $(a, b) \notin D(\lambda)$.
 Usually $T(i,j)$ is just referred as
the entry in the box $(i,j)$. 
A positive integer
$m\ge \ell(\lambda)$ will be fixed and $[m]$ will be used as a co-domain for the map $T$. We call
$[m]:=\{1, \dots, m\}$ the alphabet of the semi-standard tableau $T$.
 In this case, we will denote the \emph{set} of \emph{semi-standard} \emph{tableaux} of \emph{shape} $\lambda$ by
$SST_m(\lambda)$. 
The \emph{weight} or \emph{content} of $T$ is the nonnegative vector $\mathrm{wt}(T)=(T[1],\dots,  T[m])$, where $T[i]:=\#\{(a,b)\in D(\lambda):T(a,b)=i\}$ for $i\in[m]$, that is, $T[i]$ is  the number of occurrences of $i$ in the tableau $T$.

\subsection{Symplectic tableaux and symplectic columns}
We fix $n\in\mathbb{N}$.
We first need the definition of the \emph{parity} (\emph{swapping}) \emph{involution}.

\begin{defi}\label{def:s}For each $x\in\mathbb{Z}$, set \cite{watanabe} \begin{align}\label{s} \s(x)=\begin{cases}
x+1, &  \mbox{ if } x\notin 2\mathbb{Z}, \\
x-1, &  \mbox{ if }  x\in 2\mathbb{Z}.
\end{cases}
\end{align}
\end{defi}
Indeed $\s^2(x)=x$, for each $x\in \Z$. We call it the parity (swapping) involution.

\begin{defi}\cite{king76} \label{def:symp}A semistandard tableau $G \in SST_{2n}(\lambda)$ is said to
be symplectic if
$$G(k, 1) \ge 2k-1, \mbox{  for all $  k \in  [ \ell(\lambda)]$}.$$
Let $SpT_{2n}(\lambda)$ denote the set of all symplectic tableaux of shape $\lambda$ on the alphabet $[2n]$.
\end{defi}

\begin{prop} \label{prop:symplectic1}  
Let $G \in SST_{2n}(\lambda)$.
\begin{enumerate}
\item\cite[Proposition 2.4.2]{watanabe} If $SpT_{2n}(\lambda)\neq\emptyset$ then $\ell(\lambda)\le n$.
\item \cite{azreduction}If $G=(a_1,\dots,a_t)\in SpT_{2n}(\varpi_t)$ then $t\le n$ and
$t\le \lfloor\frac{a_t+1}{2}\rfloor$.
\item \cite[Lemma 2.4.3.]{watanabe} If $G$ is not symplectic, then there exists a unique
$i \in[2, 2n]$ such that
\begin{align}G(i, 1) < 2i -1 \mbox{ and } G(k, 1) \ge 2k -1  \mbox{ for all $k \in [1, i -1 ]$}.\label{G1}\end{align}
Moreover, we have
\begin{align}G(i - 1, 1) = 2i - 3=2(i-1)-1  \mbox{  and $G(i, 1) = 2i -2=2(i-1)$ }.\label{G2}\end{align}
\item \cite{azreduction}\label{crucialpoint}If the first column $G$ has no consecutive integers consisting of an odd number followed with an even number, that is, the first column of $G$ does not contain any interval of the form $[x<\s(x)]$, then  $G$ is symplectic.
    \item \cite{azreduction}\label{conseqcrucialpoint}$SpT_{2n}(\varpi_1)=SST_{2n}(\varpi_1)$ and $SpT_{2n}(\varpi_2)=SST_{2n}(\varpi_2)\setminus\{(1,2)\}$.
\end{enumerate}
\end{prop}

\begin{cor}\cite{azreco, azreduction}\label{symplecticcolumn} The following holds:
\begin{enumerate}
\item[(a)] For $0\le t\le n$, $SpT_{2n}(\varpi_t)\neq \emptyset$. Namely,  $SpT_{2n}(())=\{()\}$, and,  for $1\le t\le n$, $G=(1,3, \dots,2t-1)\in SpT_{2n}(\varpi_t)$ or $H=(2,4,\dots, 2t)\in SpT_{2n}(\varpi_t)$. 
\item [(b)]  $SpT_{2n}(\lambda)\neq\emptyset$ if and only if  $\ell(\lambda)\le n$.

\item [(c)] If $G'$ is obtained from $G\in  SpT_{2n}(\varpi_t)$ by suppressing
$0\le t_0\le t$ entries then $G'\in SpT_{2n}(\varpi_{t-t_0})$ and the suppressed part $G''\in SpT_{2n}(\varpi_{t_0})$.

 \item [(d)]
  Let $\lambda\in Par_\le n$,
 $$SpT_{2n}(\lambda)=\{S\in SST_{2n}(\lambda)| (S(1,1),\cdots,S(\ell(\lambda),1)\in SpT_{2n}(\varpi_{\ell(\lambda)}) \}.$$
where $S(i,1)$ indicates the entry  in row $i$ and column $1$ of $S$, for $1\le i\le \ell(\lambda)$.
 \end{enumerate}
\end{cor}

\section{The   reduction map}\label{sec:redu}
For the reader convenience this section recalls several properties of removal and reduction maps in \cite{watanabe} and \cite{azreco}.
We fix $l \in [0, 2n]$ and $\mathbf{a} = (a_1, \dots , a_l)$ a column  in $SST_{2n}(\varpi_l)$.
Recall often  $\mathbf{a}$ is regarded as a set or its reading word top to bottom.
The \emph{removal} \emph{subword }of $\mathbf{a}$ \cite{watanabe} is defined to be the subword $\mathrm{rem}(\mathbf{a})$ of  $\mathbf{a}$ obtained by the following recursive formula:

\begin{align}\label{removal}
\mathrm{rem}(\mathbf{a}) :=
\begin{cases}
\emptyset,& \mbox{ if } l\le 1,\\
\mathrm{rem}(a_l,\dots, a_{l-2})  (a_{l-1},a_l),& \mbox{ if } l\ge 2, a_l\in   2\mathbb{Z}, a_{l-1} = a_l-1,\\
 &\mbox{ and }
a_l < 2l -|\rem(a_1,\dots,a_{l-2})|-1,\\
\rem(a_1,\dots,a_{l-1})& \mbox{ otherwise}.
\end{cases}
\end{align}

\begin{defi}\cite{watanabe} \label{def:rem} For the column
$\mathbf{a}$, the new column $\mathrm{red}(\mathbf{a})$, reduction of $\mathbf{a}$, is defined to be the one obtained from $\mathbf{a}$ by removing the entries in
the set $\mathrm{rem}(\mathbf{a})$,
$$\redu( \mathbf{a})=\mathbf{a}\setminus \rem(\mathbf{a})\in SST_{2n}(\varpi_{l-|\rem(\mathbf{a})|}).$$
\end{defi}
In fact  Proposition \ref{prop:redu} below,  \cite[Section 4.3, Proposition 4.3.6]{watanabe}, shows that $\redu( \mathbf{a})\in SpT_{2n}(\varpi_{l-|\rem(\mathbf{a})|})$, where $l-|\rem(\mathbf{a})|$, satisfy further conditions.
We first recall  useful properties of removable entries in $\mathbf{a}$. In particular, we highlight the parity swapping involution $\s$ with a meaningful role in the removal of pairs and consequently in the reduction map and  in our computation of the inverse reduction map $\redu$.

In the next proposition points $(2)$ and $(3)$ together  are an alternative to the usage of Definition \ref{def:rem} to compute $\rem(\bf a)$, and providing an explicit recognition of the removal pairs.

\begin{prop}\cite[Proposition 4.2.2, 4.2.3]{watanabe}\label{prop:rem}
\begin{enumerate}

\item[(1)] If $i\in [1,l]$ and $a_j\notin \rem(\mbf{a})$, for $j\in[i,l]$ then $\rem(\mbf{a})=\rem(a_1,\dots,a_{i-1})$.

\item[(2)] for each $i\in [1,l]$, $a_i\in \rem(\mbf{a})$ if and only if one of the following holds
\begin{enumerate}
\item[(a)] $a_i$ odd, $i<l$, $a_{i+1}=a_i+1$ and $a_i<2i-|\rem(a_1,\dots,a_{i-1})|$

\item[(b)] $a_i$ even, $i>1$, $a_{i}=a_{i-1}+1$ and $a_i<2i-|\rem(a_1,\dots,a_{i-2})|-1$
\end{enumerate}

\item [(3)]  $a_i \in \rem(\mbf{a})$ if and only if $\s(a_i) \in \rem(\mbf{a})$.
Consequently, $|\rem(\mbf{a})|\in 2\mathbb{Z}$.



\end{enumerate}

\end{prop}

\begin{cor} \label{empty0} For $l \in [0, 2n]$,

\begin{enumerate}
\item[(1)]\cite[Lemma4.3.1]{watanabe} $l$ even $\Rightarrow \rem(1,2,\dots,l)=(1,2,\dots,l)=\redu(1,2,\dots,l)=\emptyset$.

\item[(2)] $l$ odd $\Rightarrow \rem(1,2,\dots,l)=\rem(1,2,\dots,l-1)= (1,2,\dots,l-1)\Rightarrow \redu(1,2,\dots,l)=\{l\}$.

\item[(3)] $\rem(\mbf{a})=\mbf{a}$ if and only if $l$ is even and $\mbf{a}=(1,2,\dots,l)$.

\item[(4)] $\redu(\mathbf{a})=\emptyset$ if and only if $l$ is even and $\mathbf{a}=(1,2,\dots,l)$.

\end{enumerate}
\end{cor}

\begin{ex}  

$\rem(6,7,8,9,10,11,12)=(11,12)$, $\redu(6,7,8,9,10,11,12)=(6,7,8,9,10)$

$\rem(1,2,3,4,5)=(1,2,3,4)$, $\redu(1,2,3,4,5)=(5)$.
\end{ex}

\begin{prop} \label{prop:redu}\cite{watanabe,nsw} Let $l\in [0,2n]$ and  $\mathbf{a}=(a_1,\dots,a_l) \in SST_{2n}(\varpi_l)$.
Then
\begin{enumerate}

\item[(1)]\cite[Proposition 4.3.6.]{watanabe} $\redu(\mathbf{a})$ is symplectic.

    \item[(2)]\cite[Proposition 4.3.7.]{watanabe} $\redu(\mathbf{a})=\mathbf{a}$ if and only if $\mathbf{a}$ is symplectic.

\item[(3)] \cite[Proposition 4.3.6., Corollary 4.4.3.]{watanabe} The reduction map $\redu$ on $SST_{2n}(\varpi_l)$ is the injective assignment

\begin{align}\redu=\redu_l:SST_{2n}(\varpi_l)&\rightarrow\bigsqcup_{\begin{smallmatrix} 0\le t\le min\{l,2n-l\}\\
l-t\in 2\mathbb{Z}
\end{smallmatrix}}SpT_{2n}(\varpi_t)\nonumber\\
\mathbf{a}&\mapsto \redu( \mathbf{a})=\mathbf{a}\setminus \rem(\mathbf{a}).
\label{redumap}
\end{align}
\end{enumerate}
\end{prop}


















\medskip

\section{The inverse  reduction map of a symplectic column by decreasing the rank by one}\label{sec:final}

In \cite[Sections 4.3,~4.4]{watanabe} the reduction map is given as composition of several maps, as below in Theorem \ref{thm:composition}, and among them combinatorial  $R$-matrices \cite{rmatrix} and a reduction map on a lower length  column. In this section we study the inverse reduction map as a composition of  maps and combinatorial R-matrices and detect in Theorem \ref{th:main} the  symplectic column of rank reduced by one on which the inverse reduction map of lower rank acts. Theorem \ref{th:main} is preceded by a lengthy proof in Proposition \ref{prop:surjective}.

\subsection{ Combinatorial R-matrices for single columns: Nakayashiki-Yamada pairing}\label{sec:Rcombinatorial}

We recall the Nakayashiki-Yamada rule or pairing to describe the combinatorial $R$-matrix for single-columns in \cite[Rule 3.10]{rmatrix} as recursively presented in \cite[Section 4.1 ]{watanabe}.

\begin{defi}Let $k, l \in [0,m]$. The combinatorial $ R$-matrix for single columns
is the map
$$R = R_{k,l} : SST_m(\varpi_k) \times SST_m(\varpi_l) \rightarrow SST_m(\varpi_l) \times SST_m(\varpi_k)$$

defined as follows. Let ${\bf a }= (a_1, \dots a_k)\in SST_m(\varpi_k)$ and ${\bf b} = (b_1, \dots, b_l) \in SST_m(\varpi_l)$
\begin{enumerate}
\item  When $k \le  l$.

 For each $r \in[1, k]$, define $i_r \in [1, l]$ inductively as follows.
 Set $i_1$ to
be the minimum $i\in [1, l] $ such that $ a_1\le b_i$; when such $i$ does not exist, we set
$i_1 := 1$. Suppose that $r \ge  2$ and we have determined $i_1, \dots , i_{r-1}$. Set $i_r$ to be the
minimum $i \in[1, l] \setminus \{i_1,
\dots , i_{r-1}\}$ such that $ a_r\le b_i$; when such $i$ does not exist,
we set $i_r := min([1,l] \setminus\{i_1, \dots, i_{r-1}\})$. Then, we set

$$R({\bf a},{\bf  b}):= ({\bf a} \sqcup {\bf b''}, {\bf b'}),$$
where
$${\bf b'} :=(b_{i_1},\dots,b_{i_k}), ~~~~ {\bf b''} :={\bf b} \setminus {\bf b'}.$$

\item  When $k \ge l. $

For each $r \in  [1, l]$, define $i_r \in [1, k]$ inductively as follows. Set $i_1$ to
be the maximum $i \in [1, k]$  such that $a_i \le  b_1$; when such $i$ does not exist, we set
$i_1 := k$. Suppose that $r \ge 2$ and we have determined $\{i_1, \dots , i_{r-1}
\}$. Set $i_r$ to be the
maximum $i\in[1, k] \setminus \{i_1,\dots,i_{r-1}\}$ such that $a_i \le b_r$; when such $i$ does not exist,
we set $i_r := max([1, k] \setminus \{i_1,\dots,i_{r-1}\})
$. Then, we set
$$R({\bf a},{\bf  b}) :=({\bf a'}, {\bf b} \sqcup {\bf a''}),$$
where

$${\bf a'}:= (a_{i_1} , \dots, a_{i_l}), ~~~~{\bf a''}:={\bf a} \setminus {\bf a'}.$$

\end{enumerate}

\end{defi}

The map $R_{kl}$ is a bijection and its inverse is $R_{lk}$ \cite[Proposition]{rmatrix}.

Given an integer $a \in [1,m]$, set
 \begin{align}a^\vee:=[1,m]\setminus\{a\}\in SST_{m}(\varpi_{m-1}).
 \end{align}

 Then $$SST_{m}(\varpi_{m-1})=\{a^\vee:a\in [m]\}$$

 When $m=2n$, $$SST_{2n}(\varpi_{2n-1})=\{a^\vee:a\in [2n]\}=\{\s(a)^\vee:a\in [2n]\}.$$

\begin{ex}Let $n=4$.
\begin{enumerate}
\item Let ${\bf a}=(5^\vee)=(1,2,3,4,6,7,8)\in SST_8(\varpi_7)$, ${\bf b}=(3,4,5)\in SST_8(\varpi_3)$
$$\begin{array}{cccccc}
\begin{matrix}
{\bf a}& {\bf b}\\
1&{}\\
{\oo \bf 2}& {}\\
{\bf 3}&{ \bf 3}\\
{ \bf 4}&{ \bf 4}\\
{}&{\oo\bf 5}\\
{6}&{}\\
{7}&{}\\
8&{}
\end{matrix}&&\overset{R_{7,3}}\longrightarrow&&\begin{matrix}
{\bf a'}& {\bf b}\sqcup {\bf a''}\\
{}&{1}\\
{}&\bf 3\\
{\oo \bf 2}& \bf 4\\
{\bf 3}&{\oo \bf 5}\\
{ \bf 4}&{}\\
{}&{6}\\
{}&{7}\\
{}&{8}\\
\end{matrix}\in SST_{2n}(\varpi_3)\times SST_{2n}(\varpi_7)
\end{array}
$$

\item Let ${\bf a}=(2)\in SST_8(\varpi_1)$, ${\bf b}=(2^\vee)=(1,3,4,5,6,7,8)\in SST_8(\varpi_7)$
$$\begin{array}{cccccc}
\begin{matrix}
{\bf a}& {\bf b}\\
{}&{1}\\
{ \bf 2}& {\bf 3}\\
{}&{ 4}\\
{ }&{  5}\\
{}&{6}\\
{}&{7}\\
{}&{8}\\
{}&{}
\end{matrix}&&\overset{R_{1,7}}\longrightarrow&&\begin{matrix}
{\bf a}\sqcup {\bf b''}& {\bf b'}\\
{1}&{}\\
{\bf 2}&{\bf 3}\\
{ 4}& {}\\
{5}&{ }\\
{6}&{}\\
{7}&{}\\
{8}&{}\\
{}&{}\\
\end{matrix}~~~~~~~~~~~~~~~~~~\in SST_8(\varpi_7)\times SST_8(\varpi_1)
\end{array}
$$
\item $\a=(3)$. $\b=(1,2)$, $R_{1~~2}(\a,\b)=((2,3),(1))$ and $R_{2~~1}((2,3),(1))=(\a,\b)$.
\end{enumerate}

\end{ex}
\subsection{The inverse reduction map as a composition of maps}
We collect some convenient properties of combinatorial  $R$- matrices on single columns.
\begin{prop}\label{prop:collect}
\begin{enumerate}

\item \cite[Lemma 4.1.5.]{watanabe} Let ${\bf a}= (a_1, \dots , a_k) \in  SST_m(\varpi_k)$ and ${\bf b} = (b_1, \dots, b_l) \in SST_m(\varpi_l)$ with
$k \le  l$. Suppose that for each $r\in [1, k]$, there exists $j_r \in [1, l]$ such that $a_r = b_{j_r }$, that is, ${\bf a} \subseteq {\bf b}$. Then, $R({\bf a}, {\bf b}) = ({\b}, {\bf a})$
and
$R(\b, \a) = (\a, \b)$.
\item Let $m\ge 1$.
\begin{enumerate}
\item \cite[Lemma 4.1.3.]{watanabe} $R(m,m^\vee) = (1^\vee, 1)$.
\item \cite[Lemma 4.1.4.]{watanabe} Let $t \in [1,m-1]$. Then,
$R(t^\vee, (1,\dots, t)) = ((1, \dots, t-1,m),m^\vee)$.
\end{enumerate}
\item \label{camel}\cite[Lemma 4.1.6]{watanabe} Let $a = (a_1, \dots , a_k) \in SST_m (\varpi_k)$, and $r, s \in [1,m]$ be such that $a_k + 1 <
r \le s$. Then, we have
\begin{align}R(s^\vee, a \sqcup[r, s]) = (a \sqcup[r-1,s-1], (r - 1)^\vee)\label{camel1}
\end{align}
and
\begin{align}R(a \sqcup [r -1, s-1], (r - 1)^\vee) = (s^\vee, a \sqcup[r, s]).\label{camel2}
\end{align}
\end{enumerate}

\end{prop}

\medskip
We now recall Proposition 4.4.1 in \cite{watanabe}. Also recall that $SpT_{2n}(\varpi_1)=SST_{2n}(\varpi_1)$ and $SpT_{2n}(\varpi_2)=SST_{2n}(\varpi_2)\setminus\{(1,2)\}$.

\begin{prop}\cite[Proposition 4.4.1]{watanabe} Let $l\in[2,2n]$ and ${\bf a}=(a_1,\dots,a_l)\in SST_{2n}(\varpi_l)$.
Let us write the following sequence of transformations,
\begin{align}\label{mapsteps}&{\bf a}=(a_1,\dots,a_l)\longrightarrow ((a_l), (a_1,\dots, a_{l-1}))\nonumber\\
&\longrightarrow (\s(a_l)^\vee, (a_1,\dots,a_{l-1}))\in SST_{2n}(\varpi_{2n-1})\times SST_{2n}(\varpi_{l-1})\nonumber\\
&\overset{R_{2n-1,l-1}}\longrightarrow ({\bf c},{\bf d}), \mbox{ for some ${\bf c}\in SST_{2n}(\varpi_{l-1})$ and  ~~${\bf d}\in  SST_{2n}(\varpi_{2n-1})$}\nonumber\\
&\overset{\redu_{l-1}\otimes id}\longrightarrow (\redu_{l-1}({\bf c}),{\bf d})\in SpT_{2n}(\varpi_{t-1})\times SST_{2n}(\varpi_{2n-1}) \quad \mbox{ for some $t\ge 1$ }\nonumber\\
&\overset{R_{t-1,2n-1}}\longrightarrow (\s(b_t)^\vee, {\bf b'}=(b_1,\dots, b_{t-1}))\in SST_{2n}(\varpi_{2n-1})\times SST_{2n}(\varpi_{t-1}),\nonumber\\
 &\mbox{ for some ${ b_t}>b_{t-1} $ and $\b'\in SST_{2n}(\varpi_{t-1})$}\nonumber\\
&\longrightarrow ((b_t),\b'=(b_1,\dots,b_{t-1})\nonumber\\
&\longrightarrow {\bf b}=(b_1,\dots,b_t).
\end{align}
Then \begin{align}{\bf b}=\begin{cases}
(1,2),& \mbox{ if  $a_l=l\in 2\Z$},\label{12}\\
\redu({\bf a}), &\mbox{ otherwise}.
\end{cases}\\
\nonumber
\end{align}
\end{prop}

Next theorem is a consequence of this proposition for $l\in [2,2n]$ with the following maps in \cite[Sections 4.3,~4.4]{watanabe} to complete the  translation of the steps in \eqref{mapsteps}:
\begin{align}{\bigvee}&:SST_{2n}(\varpi_l)\longrightarrow  \{((a), (a_1,\dots, a_{l-1}))\in SST_{2n}(\varpi_1)\times SST_{2n}(\varpi_{l-1}): a>a_{l-1}\}\nonumber\\
  &(a_1,\dots,a_{l-1},a)\mapsto((a),(a_1,\dots, a_{l-1})),\label{restrictvee1}
  \end{align}

  \begin{align}K&:SST_{2n}(\varpi_{1})\longrightarrow SST_{2n}(\varpi_{2n-1})\nonumber\\
  &~~~~~~~~~~~~~\quad\quad\quad\quad  (a)\mapsto\s(a)^\vee, \label{k1}
  \end{align}

  \begin{align}K^{-1}&:SST_{2n}(\varpi_{2n-1})\longrightarrow SST_{2n}(\varpi_1)\nonumber\\
  &~~~~~~~~~~~~~\quad\quad\qquad \quad \s(a)^\vee\mapsto(a), \mbox{ $a\in [2n]$},
  \end{align}
and, for $t\ge 1$,
\begin{align}\bigwedge&:\{((a),a_1,\dots,a_{t-1})\in SST_{2n}(\varpi_1)\times SST_{2n}(\varpi_{t-1}): a>a_{t-1}\}\longrightarrow SST_{2n}(\varpi_{t})\nonumber\\
  &~~~~~~~~~~~~~\qquad\qquad\qquad\qquad\qquad \quad\qquad\qquad\qquad\qquad((a),(a_1,\dots, a_{t-1}))\mapsto(a_1,\dots, a_{t-1},a).\label{restrictwedge1}
  \end{align}
  Due to the first point of \eqref{12}, we still need another map $\pi$ such that $\pi(1,2)=()$, and   $\pi=id$ when restricted to $ \bigsqcup_{k=0}^{2n}SST_{2n}(\varpi_k)\setminus \{(1,2)\}$:

  \begin{align}\pi&:\bigsqcup_{k=0}^{2n}SST_{2n}(\varpi_k)\longrightarrow \bigsqcup_{k=0}^{2n}SST_{2n}(\varpi_k)\nonumber\\
  &~~~~~~~~~~~~~\qquad\qquad\qquad\quad{\bf a}\mapsto\begin{cases}{\bf a}, &  \mbox{ if ${\bf a}\neq (1,2)$},\\
  (),&  \mbox{ if ${\bf a}=(1,2)$}.
  \end{cases}
  \end{align}

We now recall the theorem that asserts the reduction map as a composition of several maps \cite[Corollary 4.4.3]{watanabe}. This theorem  can be seen as a corollary of \cite[Proposition 4.4.1]{watanabe} above.
\begin{thm}\label{thm:composition}\cite[Corollary 4.4.3]{watanabe} For a fixed $n\in\mathbb{N}$, let $l\in[0,2n]$. Then, the map

\begin{align}\redu=\redu_l:SST_{2n}(\varpi_l)&\rightarrow\bigsqcup_{\begin{smallmatrix} 0\le t\le min\{l,2n-l\}\\
l-t\in 2\mathbb{Z}
\end{smallmatrix}}SpT_{2n}(\varpi_t), \quad {\bf a}\mapsto \redu({\bf a})\nonumber\\
\end{align}
is injective and is inductively defined as follows
\begin{enumerate}
\item for $l=0,1$, either $t=0$ or respectively $t=1$, in which cases $\redu=id$,
    \item for $l= 2$, if ${\bf a}=(1,2)$, $\redu(\bf a)=()$ and $t=0$, otherwise $\redu=id$ and $t=2$,
    \item for $l\ge 3$,
\begin{align}\label{decomposeredu}\redu=\pi\circ\bigwedge\circ(K^{-1},id)\circ R_{t-1,2n-1}\circ(\redu_{l-1},id)\circ R_{2n-1,l-1}\circ (K,id)\circ \bigvee,
\end{align}
\end{enumerate}
where $t-1$ denotes the length of the obtained column after the action of the reduced map $\redu_{l-1}$. The output column has length $t=t-1+1$ except when $\pi\neq id$ in which case ${\bf a}=(1,\dots,l)$, $l\in 2\Z$ and $\redu({\bf a})=()$.
\end{thm}

\begin{ex}Let $n=4$. We compute $\redu({\bf a})$ using previous theorem
\begin{enumerate}
\item Let $l=4\in [0,2n]$ and ${\bf a}=(3,4,5,6)\in SST_{2n}(\varpi_l)$.

\begin{align}& (3,4,5,6)\overset{\bigvee}\longrightarrow ((6),(345))\in   SST_{2n}(\varpi_1) \times SST_{2n}(\varpi_{l-1})\nonumber\\
&\overset{K\otimes id}\longrightarrow(\s(6)^\vee, (3,4,5))\in SST_{2n}(\varpi_{2n-1}) \times SST_{2n}(\varpi_{l-1})\nonumber
\end{align}
\begin{align}
&\overset{R_{2n-1,l-1}}\longrightarrow((2,3,4),345\sqcup\{1,6,7,8\})\in ST_{2n}(\varpi_{l-1}) \times SST_{2n}(\varpi_{2n-1})\nonumber\\
&\overset{\redu_{l-1}\otimes id}\longrightarrow ((2), (3,4,5)\sqcup\{1,6,7,8\})\in  ST_{2n}(\varpi_{t-1}) \times SST_{2n}(\varpi_{2n-1})\nonumber
\end{align}
\begin{align}
&\overset{R_{t-1,2n-1}}\longrightarrow ((s(4)^\vee),(3))\in  SST_{2n}(\varpi_{2n-1})\times SST_{2n}(\varpi_{t-1})\nonumber\\
&\overset{K^{-1}\otimes id}\longrightarrow ((4),(3))\in SST_{2n}(\varpi_{1})\times SST_{2n}(\varpi_{t-1})\nonumber\\
&\overset{\bigwedge}\longrightarrow (3,4)\in SST_{2n}(\varpi_{1+(t-1)})\nonumber\\
&\overset{\pi=id}\longrightarrow (3,4)=\redu({\bf a})\in SpT_{2n}(\varpi_2),\qquad t=2=1+(t-1),
\end{align}
and $t=2\le min\{l,2n-l\}$ and $l-t\in 2\Z$.

\item Let $l=6\in [0,2n]$ and ${\bf a}=(3,4,5,6,7,8)\in SST_{2n}(\varpi_l)$
\begin{align}& (3,4,5,6,7,8)\overset{\bigvee}\longrightarrow ((8),(34567))\in   SST_{2n}(\varpi_1) \times SST_{2n}(\varpi_{l-1})\nonumber\\
&\overset{K\otimes id}\longrightarrow(\s(8)^\vee, (3,4,5,6,7))\in SST_{2n}(\varpi_{2n-1}) \times SST_{2n}(\varpi_{l-1})\nonumber
\end{align}
\begin{align}
&\overset{R_{2n-1,l-1}}\longrightarrow((2,3,4,5,6),(3,4,5,6,7)\sqcup\{1,8\})\in SST_{2n}(\varpi_{l-1}) \times SST_{2n}(\varpi_{2n-1})\nonumber\\
&\overset{\redu_{l-1}\otimes id}\longrightarrow ((2), (3,4,5,6,7)\sqcup\{1,8\})\in  SpT_{2n}(\varpi_{t-1}) \times SST_{2n}(\varpi_{2n-1})\nonumber
\end{align}
\begin{align}
&\overset{R_{t-1},2n-1}\longrightarrow ((s(4)^\vee),(3))\in  SST_{2n}(\varpi_{2n-1})\times SpT_{2n}(\varpi_{t-1})\nonumber\\
&\overset{K^{-1}\otimes id}\longrightarrow ((4),(3))\in SST_{2n}(\varpi_{1})\times SST_{2n}(\varpi_{t-1})\nonumber
\end{align}
\begin{align}
&\overset{\bigwedge}\longrightarrow (3,4)\in SST_{2n}(\varpi_{1+t-1 })\nonumber\\
&\overset{\pi=id}\longrightarrow (3,4)\in SpT_{2n}(\varpi_t),\qquad t=2=1+(t-1).
\end{align}
and $t=2\le min\{l,2n-l\}$ and $l-t\in 2\Z$.

\end{enumerate}
\end{ex}

As we may check all these steps are reversible. 

This is conceptually very interesting because, in particular, it allows to reduce the computation  to  columns  of lower length \cite[Proposition 4.4.1]{watanabe}. Combinatorial $R$-matrices are bijections and so we could use the inverse of these maps to compute the inverse reduction.
However one still  needs to compute the inverse reduction of shorter symplectic columns  and thereby it is useful to  have at hand the inverse computed in \cite{azreduction}.

We shall  next use the previous inductive definition of the injective map $\redu$ in Theorem \ref{thm:composition} to prove that is also surjective. Indeed we already know from \cite{azreduction} the explicit inverse reduction map. The proof of the surjectivity of the reduction map in the framework of those composition of maps will detect the shorter symplectic column that should conside we consider in the inverse of the reduction map as a composition of those aforesaid maps. The final goal is to avoid the sequence of composition of maps and keep the essential information.

\begin{prop}\label{prop:surjective} The reduction map $\redu$ is  surjective and therefore $\redu^{-1}$ does exist.
\end{prop}

\begin{proof}Since $SpT_{2n}(\varpi_1)=SST_{2n}(\varpi_1)$ and $SpT_{2n}(\varpi_2)=SST_{2n}(\varpi_2)\setminus\{(1,2)\}$, clearly from previous theorem, the map $\redu$ is surjective for $l=0,1,2$. Let $l\ge 3$, and by induction assume that the map $\redu$ is surjective for $l-1$.

Let   $0\le t\le min\{l,2n-l\}$ such that
$l-t\in 2\mathbb{Z}$.
Given ${\bf b}\in SpT_{2n}(\varpi_t)$ let us prove that $\b$ is reached by the map $\redu$ in $SST_{2n}(\varpi_l)$, that is, there exists $\a\in SST_{2n}(\varpi_l)$ such that $\redu(\a)=\b$.

If $t=0$ and $l\in 2\Z$, then ${\bf b}=()$ and by Corollary \ref{empty0},  $\redu(1,2,\dots, l)=()$ with  $l\in 2\Z$, and $\a=(1,2,\dots,l)\in SST_{2n}(\varpi_l)$. 

\bigskip

If $t= 1$ and $2\le l-1\in 2\Z$, let $\b=(b)\in SpT_{2n}(\varpi_1)$ and define
\begin{align}&\a=
\begin{cases}[1,l_1]\sqcup(b)\sqcup [ b+1, b+l_{2}],
  &
 \mbox{if } b\in 2\mathbb{Z}\\
 [1,l_1])\sqcup(b)\sqcup[ b+1+1, b+1+l_{2}],
 &\mbox{if } b\notin 2\mathbb{Z}
\end{cases},\label{t=1}
\end{align}

and
\begin{align}
&l_1=
\begin{cases} min\{b-2, l-1\},& \mbox{if } b\in 2\mathbb{Z}\\
min\{ b-1, l-1\},&\mbox{if } b\notin 2\mathbb{Z},\\
\end{cases}
\qquad \mbox{ and }  l_{2}=l-1-l_1\in 2\Z.
\end{align}

Henceforth, $\a \in SST_{2n}(\varpi_l)$ and $\redu(\a)=\b$.

So far, one has $\redu_l^{-1}()=[1,l]$ with $4\le l\in 2\Z$, and $\redu_l^{-1}((b))=\a\in\SST(\varpi_l)$, with $3\le l\notin 2\Z$, such that
if \begin{align}\label{induczerone}\hat\a:=[1,l-1]=\redu_{l-1}^{-1}()
\end{align}
then
\begin{align}\label{inducxx}\a=\begin{cases}(\hat \a,b),& \mbox{ if $\hat \a\subseteq[1,b-1]$},\\
\hat \a\setminus \{b+1\}\sqcup [l,l+1],& \mbox{ if $b\notin 2\Z$ and  $b\in\hat \a=[1,l-1]$},\\
\hat \a\setminus \{b-1\}\sqcup  [l,l+1],& \mbox{ if $b\in 2\Z$ and  $b\in\hat \a=[1,l-1]$},\\
\end{cases}
\end{align}
and $\a\supseteqq \b$.

Note that \eqref{inducxx} coincides with \eqref{t=1}.

Let $2\le b\in 2\Z$.

 If $\hat \a\subseteq[1,b-1]$ holds in \eqref{inducxx}, $l-1<b-1$ because $l-1\in 2\Z$ and $b-1\notin 2\Z$.
It implies that $l-1\le b-2$ and $l_1=l-1$ and $l_2=0$ in \eqref{t=1}. This means for $b\in 2\Z$, $\a=\hat\a\sqcup (b)$ in \eqref{t=1}.

If $b\in\hat \a=[1,l-1]$ holds in \eqref{inducxx}, then $2\le b\le l-1$ and in \eqref{inducxx} one has
$$ \a=\hat \a\setminus \{b-1\}\sqcup [l,l+1]=[1,b-2]\sqcup [b,l-1]\sqcup[l,l+1]=[1,b-2]\sqcup (b)\sqcup [b+1,l+1].$$.

In \eqref{t=1}, one has $l_1=b-2$,  $l_2=l-1-b+2=l-b+1$, and also
$$ \a=[1,b-2]\sqcup (b)\sqcup[b+1, b+l-b+1]=[1,b-2]\sqcup (b)\sqcup[b+1, l+1].$$

Let $ b\notin 2\Z$.

 If $\hat \a\subseteq[1,b-1]$ holds in \eqref{inducxx}, then  $l-1\le b-1$.
It implies that  $l_1=l-1$ and $l_2=0$ in \eqref{t=1}. This means for $b\notin 2\Z$, $\a=\hat\a\sqcup (b)$ in \eqref{t=1}.

If $b\in\hat \a=[1,l-1]$ holds in \eqref{inducxx}, then $ b< l-1\Leftrightarrow b+1\le l-1$ because $b\notin 2\Z$ and $l-1\in 2\Z$. in \eqref{inducxx} one has
$$ \a=\hat \a\setminus \{b+1\}\sqcup [l,l+1]=[1,b]\sqcup [b+2,l-1]\sqcup[l,l+1]=[1,b-1]\sqcup (b)\sqcup [b+2,l+1].$$.

In \eqref{t=1}, one has $l_1=b-1$,  $l_2=l-1-b+1=l-b$, and also

$$ \a=[1,b-1])\sqcup(b)\sqcup[ b+1+1, b+1+l-b]=[1,b-1]\sqcup (b)\sqcup [b+2,l+1].$$

\bigskip

Let $t\ge 2$. We now apply to $\b=(b_1,\dots,b_t)\in SpT_{2n}(\varpi_t)$ the reverse procedure in \eqref{mapsteps}.

One has
\begin{align}&\b=(b_1,\dots,b_t)\in SpT_{2n}(\varpi_t)\rightarrow ((b_t),(b_1,\dots,b_{t-1}))\in SpT_{2n}(\varpi_1)\times SpT_{2n}(\varpi_{t-1})\nonumber\\
&\rightarrow (\s(b_t)^\vee,\b'=(b_1,\dots,b_{t-1}))\in SST_{2n}(\varpi_{2n-1})\times SpT_{2n}(\varpi_{t-1})\nonumber\\
&\overset{R_{2n-1,t-1}}\longrightarrow(\c', \d) \mbox{ for some $(\c',\d)\in \SpT(\varpi_{t-1})\times \SST(\varpi_{2n-1})$.}\label{needsproof}
\end{align}

We now prove \eqref{needsproof}, that is, we exhibit $(\c', \d)\in  \SpT(\varpi_{t-1})\times \SST(\varpi_{2n-1})$, such that \begin{align}R_{2n-1,t-1}(\s(b_t)^\vee,\b')=(\c', \d).\label{goal}
\end{align}
The proof depends on the parity of the last component $b_t$ of $\b$: either $\s(b_t)=b_t+1$ if $b_t\notin 2\Z$, or
$\s(b_t)=b_t-1$ if $b_t\in 2\Z$.

\medskip

\textbf{Case 1.}  $b_t\notin 2\Z\Leftrightarrow\s(b_t)=b_t+1$.

If
$b_t\notin 2\Z$ then  $ \s(b_t)=b_t+1\in 2\Z$ and $(\s(b_t)^\vee,\b')=((b_t+1)^\vee,\b')$ with $\b'=(b_1,\dots,b_{t-1})$ $\subseteq [1,b_{t-1}]\in SpT_{2n}(\varpi_{t-1})$.
Hence $\b'\subseteq  (b_t+1)^\vee$ and by Proposition \ref{prop:collect}, $(1)$,
\begin{align}{R_{2n-1, t-1}}((b_t+1)^\vee,\b')=(\b', (b_t+1)^\vee)=(\c',\d)\in \SpT(\varpi_{t-1})\times \SST(\varpi_{2n-1})\label{snale1}.
\end{align}

 Note $t-1\le min\{l-1, 2n-(l-1)\}$ and $l-t\in 2\Z\Leftrightarrow l-1-(t-1)\in 2\Z$. Then
by induction on $l$, one has
\begin{align}\label{induc1}
\b'=\redu_{l-1}(\hat \a), \mbox{ for some $\hat\a\in \SST(\varpi_{l-1})$ with $\b'=(b_1,\dots,b_{t-1})\subseteq\hat\a$.}
\end{align}
That is, $\redu_{l-1}$ is injective and surjective, and  $\redu^{-1}_{l-1}$ does exist,
\begin{align}\redu^{-1}_{l-1}(\b')=\redu^{-1}_{l-1}(b_1,\dots,b_{t-1})=\hat\a=(a_1,\dots,a_{l-1})\in \SST(\varpi_{l-1}) \mbox{ with $\b'=(b_1,\dots,b_{t-1})\subseteq\hat\a$}.\label{snale2}
\end{align}

Since $t-1\le l-1$, by induction on $t\ge 1$, we then may write,
  for some column $\a''\sqcup \{b_{t-1}\}:=\hat\a\cap[1,b_{t-1}]$, with $b_0:=0$,
\begin{align}\label{induc2}\hat\a=\begin{cases}\a''\sqcup \{b_{t-1}\}\sqcup [b_{t-1}+1,b_{t-1}+s],~~\mbox{ for some $0\le s\in 2\Z$}& \mbox{ if $b_{t-1}\in 2\Z$}\\
\a''\sqcup \{b_{t-1}\}\sqcup [b_{t-1}+1+1,b_{t-1}+1+s],~~\mbox{ for some $0\le s\in 2\Z$}& \mbox{ if $b_{t-1}\notin 2\Z$},
\end{cases}
\end{align}
such that $s+\ell(\hat\a\cap[1,b_{t-1}])=l-1$.

Note if $s=0$, it means the intervals in \eqref{induc2} are  empty. If the intervals are not empty then $2\le s\in 2\Z$.

Hence, from \eqref{snale1}, \eqref{snale2},
\begin{align}&(\redu^{-1}_{l-1}\otimes id){R_{2n-1, t-1}}((b_t+1)^\vee,\b')=(\redu^{-1}_{l-1}\otimes id)(\b',(b_t+1)^\vee)\nonumber \\
&=(\hat\a, (b_t+1)^\vee),\nonumber \\
& \mbox{ with $\a\in \SST(\varpi_{l-1})$ satisfying \eqref{induc2}}. \label{cow}
\end{align}

From \eqref{induc2} and $ \s(b_t)=b_t+1\in 2\Z$, we have to analyse two cases

\begin{itemize}

\item $ \hat \a=(a_1,\dots, a_{l-1})\subseteq [1,b_t-1].$
\end{itemize}
From \eqref{snale1} and \eqref{cow}
$$R_{l-1,2n-1}(\redu^{-1}_{l-1}\otimes id){R_{2n-1,t-1}}((b_t+1)^\vee,\b')=R_{l-1,2n-1}(\hat\a, (b_t+1)^\vee)=( (b_t+1)^\vee,\hat\a)\Rightarrow$$
$$(K\otimes id)R_{l-1,2n-1}(\redu^{-1}_{l-1}\otimes id){R_{2n-1,t-1}}((b_t+1)^\vee,\b')=(K\otimes id)( (b_t+1)^\vee,\hat\a)$$
$$=(K\otimes id)( \s(b_t)^\vee,\hat\a) =(b_t,\hat\a)\longrightarrow (\hat\a, b_t)\in SST_{2n}(\varpi_l).$$

Hence, for $b_t\notin 2\Z$ and $ \hat \a\subseteq [1,b_t-1],$

\begin{align}\label{main1}\redu^{-1}(\b)=\a=(\hat\a, b_t)\supseteq \b,  \mbox{ where $\hat \a=\redu^{-1}_{l-1}(b_1,\dots,b_{t-1})$}.
\end{align}
\begin{itemize}
\medskip
\item  $b_t\in\hat \a$.
\end{itemize}
Note $b_{t-1}+1\le b_t$ with $b_{t-1}\in 2\Z$ and
$b_{t-1}+1<b_t$ with $b_{t-1}\notin 2\Z$. Thus $b_{t-1}+1+1\le b_t$ with $b_{t-1}\notin 2\Z$.
On the other hand  since  $b_t\in\hat \a$, from \eqref{induc2}, $b_t\in  [b_{t-1}+1,b_{t-1}+s]$, with $2\le s\in 2\Z$, if $b_{t-1}\in 2\Z$; and $ b_t\in [b_{t-1}+1+1,b_{t-1}+1+s]$,~~with $2\le s\in 2\Z$, if $b_{t-1}\notin 2\Z$. Additionally, these intervals start with an odd number and terminate with an even number, and are closed under the parity involution $\s$. Thus since $b_t$ is odd, $b_t, b_t+1$   are together in the same  interval.

Hence $[b_t,b_t+1]\subseteq \hat\a\setminus (\a''\sqcup\{b_{t-1}\})$ which implies $[b_t,b_t+1]\subseteq [b_{t-1}+1,b_{t-1}+s]$. Let $\a'''\sqcup\{b_t\}:=\hat\a\cap[1,b_t]$. We may then rewrite \eqref{induc2} as

\begin{align}\label{induc3}\hat\a=\begin{cases}\a'''\sqcup \{b_t\}\sqcup [b_{t}+1,b_{t-1}+s],~~\mbox{ for some $2\le s\in 2\Z$}& \mbox{ if $b_{t-1}\in 2\Z$}\\
\a'''\sqcup \{b_t\}\sqcup\ [b_{t}+1,b_{t-1}+1+s],~~\mbox{ for some $2\le s\in 2\Z$}& \mbox{ if $b_{t-1}\notin 2\Z$},
\end{cases}
\end{align}
such that $s+\ell(\hat\a\cap[1,b_{t-1}])=l-1$.

\begin{obs} Note $\hat\a\cap[1,b_t]=\hat\a\cap[1,b_{t-1}]\sqcup \hat\a\cap[b_{t-1}+1,b_t]$,  and considering  \eqref{induc2} with $b_{t-1}\in 2\Z$,  $$\hat \a=\hat\a\cap[1,b_{t-1}]\sqcup [b_{t-1}+1,b_{t-1}+s]=\hat\a\cap[1,b_{t-1}]\sqcup [b_{t-1}+1,b_t]\sqcup[b_t+1,b_{t-1}+s]=\hat\a\cap[1,b_t]\sqcup[b_t+1,b_{t-1}+s].$$
For $b_{t-1}\notin 2\Z$ is similar.
\end{obs}

From \eqref{cow},
one has
\begin{align}R_{l-1,2n-1}(\redu^{-1}_{l-1}\otimes id){R_{2n-1,t-1}}((b_t+1)^\vee,\b')=R_{l-1,2n-1}(\hat\a, (b_t+1)^\vee).\label{warm}
\end{align}

 \emph{If $b_{t-1}\in 2\Z$}, from Proposition \ref{prop:collect}, \eqref{camel}, \eqref{camel2}, and \eqref{induc3}

\begin{align}&\eqref{warm}\Leftrightarrow R_{l-1,2n-1}(\hat\a, (b_t+1)^\vee)=R_{l-1,2n-1}(\a'''\sqcup \{b_t\}\sqcup[b_t+1,b_{t-1}+s], (b_t+1)^\vee)\nonumber\\
&=((b_{t-1}+s+1)^\vee, \a'''\sqcup \{b_t\}\sqcup [b_t+2,b_{t-1}+s+1]) \nonumber\\
&=(\s(b_{t-1}+s+2)^\vee,\a'''\sqcup \{b_t\}\sqcup [b_t+2,b_{t-1}+s+1]), \nonumber\\
&\mbox{ by $b_{t-1}+s+1\notin 2\Z\Rightarrow \s(b_{t-1}+s+2)=b_{t-1}+s+1$}, \label{vulture}
\end{align}
and, therefore, it follows
\begin{align*}&(K\otimes id)R_{l-1,2n-1}(\redu^{-1}_{l-1}\otimes id){R_{2n-1,t-1}}((b_t+1)^\vee,\b')\\
&=(K\otimes id)R_{l-1,2n-1}(\hat\a, (b_t+1)^\vee)=(K\otimes id)R_{l-1,2n-1}(\a''\sqcup \{b_t\}\sqcup [b_t+1,b_{t-1}+s], (b_t+1)^\vee), \mbox{ by \eqref{induc3}}\\
&=(K\otimes id)(\s(b_{t-1}+s+2)^\vee,\a'''\sqcup \{b_t\}\sqcup [b_t+2,b_{t-1}+s+1])\\
&=(b_{t-1}+s+2,\a'''\sqcup \{b_t\}\sqcup [b_t+2,b_{t-1}+s+1]),   \mbox{ by \eqref{vulture}}\ \\
&\longrightarrow (\a'''\sqcup \{b_t\}\sqcup [b_t+2,b_{t-1}+s+1],b_{t-1}+s+2)=\a'''\sqcup \{b_t\}\sqcup [b_t+2,b_{t-1}+s+2]\in SST_{2n}(\varpi_l)
\end{align*}

Hence, for $b_{t-1}\in 2\Z~,~~~b_t\notin 2\Z$ and $b_t\in\hat \a$,
\begin{align}&\redu^{-1}(\b)=\a=\a'''\sqcup \{b_t\}\sqcup [b_t+2,b_{t-1}+s+2]\supseteqq \b. \nonumber 
\end{align}
In other words,
\begin{align}
&\redu^{-1}(\b)=\a, \mbox{ where}\nonumber\\
&\a=\hat\a\setminus\{b_t+1\}\sqcup [b_{t-1}+1+s, b_{t-1}+2+s]\in SST_{2n}(\varpi_l),\nonumber\\
& \mbox{ with $2\le s\in 2\Z$ in \eqref{induc2} in the definition of }\nonumber\\
&\mbox{$\hat\a:=\redu^{-1}(b_1,\dots,b_{t-1})$. }\label{rankrelation1}
\end{align}

\begin{obs}For $b_{t-1}\in 2\Z$, $[b_{t}+1,b_{t-1}+s]$ with $s\ge 2$, we  have $b_t+1\le b_{t-1}+s$ for an $s\ge 2$.
Then $b_t+2\le b_{t-1}+s+1<b_{t-1}+s+2$.
\end{obs}
\medskip

\emph{If $b_{t-1}\notin 2\Z$}, the proof is similar. From Proposition \ref{prop:collect}, \eqref{camel},  \eqref{camel2}, and \eqref{induc3}

\begin{align}&\eqref{warm}\Leftrightarrow R_{l-1,2n-1}(\a, (b_t+1)^\vee)=R_{l-1,2n-1}(\a'''\sqcup \{b_t\}\sqcup [b_t+1,b_{t-1}+s+1], (b_t+1)^\vee)\nonumber\\
&=((b_{t-1}+s+2)^\vee, \a'''\sqcup \{b_t\}\sqcup [b_t+2,b_{t-1}+s+2])\nonumber\\
&=(\s(b_{t-1}+s+3)^\vee,\a'''\sqcup \{b_t\}\sqcup [b_t+2,b_{t-1}+s+2]),\nonumber \\
&\mbox{ by $b_{t-1}+s+2\notin 2\Z\Rightarrow \s(b_{t-1}+s+3)=b_{t-1}+s+2$}, \label{tiger}
\end{align}
and
\begin{align*}&(K\otimes id)R_{l-1,2n-1}(\redu^{-1}_{l-1}\otimes id){R_{2n-1,t-1}}((b_t+1)^\vee,\b')\\
&=(K\otimes id)R_{l-1,2n-1}(\hat\a, (b_t+1)^\vee)=(K\otimes id)R_{l-1,2n-1}(\a'''\sqcup \{b_t\}\sqcup [b_t+1,b_{t-1}+s+1], (b_t+1)^\vee)\\
&=(K\otimes id)(\s(b_{t-1}+s+3)^\vee,\a'''\sqcup \{b_t\}\sqcup [b_t+2,b_{t-1}+s+2]) \mbox{ by \eqref{tiger}}\\
&=(b_{t-1}+s+3,\a'''\sqcup \{b_t\}\sqcup [b_t+2,b_{t-1}+s+2])\\
&\longrightarrow (\a'''\sqcup [b_t+2,b_{t-1}+s+2],b_{t-1}+s+3)=\a'''\sqcup \{b_t\}\sqcup [b_t+2,b_{t-1}+s+3]\in SST_{2n}(\varpi_l)
\end{align*}

 Hence, for $b_{t-1}\notin 2\Z,~~ b_t\notin 2\Z$ and $b_t\in\hat\a$,
 \begin{align}&\redu^{-1}(\b)=\a=\a'''\sqcup \{b_t\}\sqcup [b_t+2,b_{t-1}+s+3]\supseteqq \b,\nonumber\\
 &\mbox{ and }\nonumber\\
 &\a=\hat\a\setminus \{b_t+1\}\sqcup[b_{t-1}+s+2,b_{t-1}+s+3]\in\SST(\varpi_l),\nonumber\\
& \mbox{ with $2\le s\in 2\Z$ in  \eqref{induc3} in the definition of $\hat\a=\redu^{-1}(b_1,\dots,b_{t-1})$.}\label{main2}
 \end{align}


\bigskip

\textbf{Case 2.}  $b_t\in 2\Z\Leftrightarrow \s(b_t)=b_t-1$.

If $b_t\in 2\Z$,  then $b_t\in 2\Z\Rightarrow \s(b_t)=b_t-1\notin 2\Z$ and $(\s(b_t)^\vee,\b)=((b_t-1)^\vee,\b')$ with $\b'=$ $(b_1,\dots,b_{t-1})$ $\in SpT_{2n}(\varpi_{t-1})$.

We have two main  cases: either $b_{t-1}<b_t-1$ or $b_{t-1}=b_t-1$.
\begin{enumerate}
\item $b_{t-1}<b_t-1$.

If $b_{t-1}<b_t-1$, then $b_{t-1}\le b_t-2$, if $b_{t-1}\in 2\Z$,
and $b_{t-1}< b_t-2\Leftrightarrow b_{t-1}+2< b_t\Leftrightarrow b_{t-1}+2\le  b_t-1$ if $b_{t-1}\notin 2\Z$.

Hence, $\b'\subseteq  (b_t-1)^\vee$ and by Proposition \ref{prop:collect}, $(1)$,
\begin{align}{R_{2n-1,t-1}}((b_t-1)^\vee,\b')=(\b', (b_t-1)^\vee)=(\c',\d)\in \SpT(\varpi_{t-1})\times \SST(\varpi_{2n-1}). \label{snale3}\end{align}

Then
by induction on $l\ge 3$, one has
\begin{align}\label{induc1new}
\b'=\redu_{l-1}(\hat\a), \mbox{ for some $\hat\a\in \SST(\varpi_{l-1})$ such that $\b'=(b_1,\dots,b_{t-1})\subseteq\hat\a$}
\end{align}
and  by induction on $t\ge 1$ we may write for some $\a''\sqcup \{b_{t-1}\}=\hat \a\cap  [1,b_{t-1}]$, with $b_0:=0$,
\begin{align}\label{induc2new}\hat\a=\begin{cases}\a''\sqcup \{b_{t-1}\}\sqcup [b_{t-1}+1,b_{t-1}+s],~~\mbox{ for some $0\le s\in 2\Z$}& \mbox{ if $b_{t-1}\in 2\Z$}\\
\a''\sqcup \{b_{t-1}\}\sqcup [b_{t-1}+1+1,b_{t-1}+1+s],~~\mbox{ for some $0\le s\in 2\Z$}& \mbox{ if $b_{t-1}\notin 2\Z$}
\end{cases},
\end{align}
such that $s+\ell(\hat\a\cap[1,b_{t-1}])=l-1$.

 Again note if $s=0$ the intervals are empty in \eqref{induc2new}. If they are not empty, $s\ge 2$.

Hence
\begin{align}&(\redu^{-1}_{l-1}\otimes id){R_{2n-1,t-1}}((b_t-1)^\vee,\b')=(\hat\a, (b_t-1)^\vee) \nonumber\\
& \mbox{ for some $\hat\a\in \SST(\varpi_{l-1})$ satisfying  \eqref{induc2new}}.
\end{align}

We have to analyse two cases

\begin{itemize}
\item $ \hat\a=(a_1,\dots, a_{l-1})\subset [1,b_t-2]$.
\end{itemize}
$$R_{l-1,2n-1}(\redu^{-1}_{l-1}\otimes id){R_{2n-1,t-1}}((b_t-1)^\vee,\b')=R_{l-1,2n-1}(\hat\a, (b_t-1)^\vee)=( (b_t-1)^\vee,\hat\a)\Rightarrow$$
$$(K\otimes id)R_{l-1,2n-1}(\redu^{-1}_{l-1}\otimes id){R_{2n-1,t-1}}((b_t-1)^\vee,\b')=(K\otimes id)( (b_t-1)^\vee,\hat\a)$$
$$=(K\otimes id)( \s(b_t)^\vee,\hat\a) =(b_t,\hat\a)\longrightarrow (\hat\a, b_t)\in SST(\varpi_l).$$

Hence for $b_t\in 2\Z$, $b_{t-1}<b_t-1$, and $ \hat\a\subset [1,b_t-2]$.

\begin{align}\redu^{-1}(\b)=\a=(\hat\a, b_t), \mbox{ where $\hat\a=\redu^{-1}(b_1,\dots,b_{t-1})$}\label{main3}
.\end{align}

\bigskip
\begin{itemize}
\item  $ b_t-1\in \hat \a\Leftrightarrow \s(b_t)\in \hat\a$.
\end{itemize}

By assumption, $b_t\in 2\Z
 \Leftrightarrow b_t-1\notin 2\Z$, and thus
 \begin{align}b_{t-1}<b_t-1\Leftrightarrow b_{t-1}+1\le b_t-1\Leftrightarrow b_{t-1}+1+1\le b_t.\label{rino}
 \end{align}

 When $b_{t-1}\in 2\Z$, since $ b_t-1\in \hat \a$, it means in \eqref{induc2new} that $s\ge 2$. Therefore $[b_t-1,b_t]\subseteq [b_{t-1}+1,b_{t-1}+s]$ with $s\ge 2$ in \eqref{induc2new}.

 When $b_{t-1}\notin 2\Z$ then $b_{t-1}+1+1\notin 2\Z$ and in this case one has $b_{t-1}+1+1< b_t$ and $b_{t-1}+1+1\le  b_t-1$ with $ b_t-1\in \hat \a$. Thus, it means in \eqref{induc2new} that $s\ge 2$. Therefore,

 $$[b_t-1, b_t]\subseteq [b_{t-1}+1+1,b_{t-1}+1+s] \mbox{ with  $s\ge 2$  in \eqref{induc2new}}.$$

\medskip
Since $b_{t-1}\le b_t-2$, from \eqref{rino}, let $\a''':=\hat\a\cap[1,b_t-2]$, and
we may then rewrite \eqref{induc2new} as

\begin{align}\label{induc4}\hat\a=\begin{cases}\a'''\sqcup [b_{t}-1,b_{t-1}+s],~~\mbox{ for some $2\le s\in 2\Z$}& \mbox{ if $b_{t-1}\in 2\Z$}\\
\a'''\sqcup [b_{t}-1,b_{t-1}+1+s],~~\mbox{ for some $2\le s\in 2\Z$}& \mbox{ if $b_{t-1}\notin 2\Z$},
\end{cases}
\end{align}
such that $s+\ell(\hat\a\cap[1,b_{t-1}])=l-1$.
\begin{obs}With $b_{t-1}\le b_t-2$, $\hat\a\cap[1,b_t-2]=\hat\a\cap[1,b_{t-1}]\sqcup \hat\a\cap[b_{t-1}+1,b_t-2]$. Considering \eqref{induc2new} with $b_{t-1}\in 2\Z$, one has 
$$\hat\a=\hat\a\cap [1,b_{t-1}]\sqcup [b_{t-1}+1,b_{t-1}+s]=\hat\a\cap [1,b_{t-1}]\sqcup [b_{t-1}+1,b_t-2]\sqcup [b_{t}-1,b_{t-1}+s]$$
$$=\hat\a\cap[1,b_t-2]\sqcup [b_{t}-1,b_{t-1}+s].$$
For $b_{t-1}\notin 2\Z$ the proof is similar.
\end{obs}

One has \begin{align}&R_{l-1,2n-1}(\redu^{-1}_{l-1}\otimes id){R_{2n-1,t-1}}((b_t-1)^\vee,\b')=\nonumber\\
&R_{l-1,2n-1}(\redu^{-1}_{l-1}\otimes id)(\b', (b_t-1)^\vee)=R_{l-1,2n-1}(\hat\a,(b_t-1)^\vee)\nonumber\\
& \mbox{ for some $\hat\a\in \SST(\varpi_{l-1})$ satisfying \eqref{induc4}}.\label{pig}
\end{align}

\bigskip
\emph{If $b_{t-1}\in 2\Z$}, from Proposition \ref{prop:collect}, \eqref{camel}, \eqref{camel2} and \eqref{induc4}

\begin{align}&\eqref{pig}\Leftrightarrow R_{l-1,2n-1}(\hat\a, (b_t-1)^\vee)=R_{l-1,2n-1}(\a''' \sqcup [b_t-1,b_{t-1}+s], (b_t-1)^\vee)\nonumber\\
&=((b_{t-1}+s+1)^\vee, \a''' \sqcup[b_t,b_{t-1}+s+1])\nonumber\\
&=(\s(b_{t-1}+s+2)^\vee,\a''' \sqcup [b_t,b_{t-1}+s+1]), \nonumber\\
&\mbox{ by $b_{t-1}+s+1\notin 2\Z\Rightarrow \s(b_{t-1}+s+2)=b_{t-1}+s+1$}, \label{vulture2}
\end{align}
and
\begin{align*}&(K\otimes id)R_{l-1,2n-1}(\redu^{-1}_{l-1}\otimes id){R_{2n-1,t-1}}((b_t-1)^\vee,\b')\\
&=(K\otimes id)R_{l-1,2n-1}(\hat\a, (b_t-1)^\vee)=(K\otimes id)R_{l-1,2n-1}(\a''' \sqcup [b_t-1,b_{t-1}+s], (b_t-1)^\vee)\\
&=(K\otimes id)(\s(b_{t-1}+s+2)^\vee,\a''' \sqcup [b_t,b_{t-1}+s+1])\\
&=(b_{t-1}+s+2,\a''' \sqcup [b_t,b_{t-1}+s+1]),   \mbox{ by \eqref{vulture2}}\ \\
&\longrightarrow (\a''' \sqcup [b_t,b_{t-1}+s+1],b_{t-1}+s+2)\\
&=\a''' \sqcup [b_t,b_{t-1}+s+2]\in SST_{2n-1}(\varpi_l)
\end{align*}

Hence, for $~b_t\in 2\Z,~~~~b_{t-1}\in 2\Z$ and $b_{t-1}<b_t-1$ and $\s(b_t)\in \hat\a$,
\begin{align*}&\redu^{-1}(\b)=\a=\a''' \sqcup [b_t,b_{t-1}+s+2]\nonumber\\
&=\a''' \sqcup\{b_t\} \sqcup[b_t+1,b_{t-1}+s+2]\supseteqq\b,\nonumber\\
&\a=\hat\a\setminus\{b_t-1\}\sqcup [b_{t-1}+s+1,b_{t-1}+s+2], \mbox{ for some $2\le s\in 2\Z$ as in \eqref{induc2new}.}\\
&\hat \a=\redu^{-1}(b_1,\dots,b_{t-1}).
\end{align*}

\bigskip
\emph{If $b_{t-1}\notin 2\Z$}, from Proposition \ref{prop:collect}, \eqref{camel}, \eqref{camel2} and \eqref{induc4}

\begin{align}&\eqref{pig}\Leftrightarrow R_{l-1,2n-1}(\hat\a, (b_t-1)^\vee)=R_{l-1,2n-1}(\a''' \sqcup [b_t-1,b_{t-1}+s+1], (b_t-1)^\vee)\nonumber\\
&=((b_{t-1}+s+2)^\vee, \a''' \sqcup [b_t,b_{t-1}+s+2])\nonumber\\
&=(\s(b_{t-1}+s+3)^\vee,\a''' \sqcup [b_t,b_{t-1}+s+2]), \nonumber\\
&\mbox{ by $b_{t-1}+s+2\notin 2\Z\Rightarrow \s(b_{t-1}+s+3)=b_{t-1}+s+2$}, \label{vulture3}
\end{align}
and
\begin{align*}&(K\otimes id)R_{l-1,2n-1}(\redu^{-1}_{l-1}\otimes id){R_{2n-1,t-1}}((b_t-1)^\vee,\b')\\
&=(K\otimes id)R_{l-1,2n-1}(\a, (b_t-1)^\vee)=(K\otimes id)R_{l-1,2n-1}(\a''' \sqcup [b_t-1,b_{t-1}+s+1], (b_t-1)^\vee)\\
&=(K\otimes id)(\s(b_{t-1}+s+3)^\vee,\a''' \sqcup[b_t,b_{t-1}+s+2])\\
&=(b_{t-1}+s+3,\a''' \sqcup[b_t,b_{t-1}+s+2]),   \mbox{ by \eqref{vulture3}}\ \\
&\longrightarrow (\a'''\sqcup [b_t,b_{t-1}+s+2],b_{t-1}+s+3)=\a''' \sqcup[b_t,b_{t-1}+s+3]\in SST_{2n-1}(\varpi_l)
\end{align*}

Hence, for $b_{t-1}\notin 2\Z,~~~b_t\in 2\Z$, and  $b_{t-1}\le b_t-2$, and $\s(b_t)\in \hat\a$,

\begin{align}&\redu^{-1}(\b)=\a=\a'''\sqcup \{b_t\}\sqcup [b_t+1,b_{t-1}+s+3]\supseteqq \b,\nonumber\\
&\a=\hat \a\setminus \{b_t-1\}\sqcup[b_{t-1}+s+2,b_{t-1}+s+3],  \mbox{ for some $2\le s\in 2\Z$ as in \eqref{induc2new},}\nonumber\\
&\hat \a=\redu^{-1}(b_1,\dots,b_{t-1}).\label{main4}
\end{align}

\bigskip
\item  $b_{t-1}=b_t-1\notin 2\Z$

Recall $\b=(b_1,\dots,b_t)\in SpT_{2n}(\varpi_t)$ and $\b'=(b_1,\dots,b_{t-1}=b_t-1)\in SpT_{2n}(\varpi_{t-1})$.
Unlike the previous cases we  have
$$\b'\subsetneq \s(b_t)^\vee=(b_t-1)^\vee.$$

 Set
$r :=max\{i \in [1,t-1] : b_{t-i} = b_t-i\}$ and  $\b'' := (b_1, \dots , b_{t-r-1})$. Note $r\ge 1$ and \begin{align}b_{t-r-1}<b_t-r-1<b_t-r=b_{t-r}.\label{ineq}
\end{align}

Then, we have $\b = \b'' \sqcup [b_t-r, b_t],$  $\b'=\b'' \sqcup[b_t-r, b_t-1]\in \SpT(\varpi_{t-1})$,
$\s(b_t) = b_t-1$, and

from Proposition \ref{prop:collect}, \eqref{camel}, \eqref{camel1},

\begin{align*}&\eqref{goal}\Leftrightarrow R_{2n-1,t-1}((b_t-1)^\vee, \b') = R_{2n-1,t-1}((b_t-1)^\vee, \b'' \sqcup[b_t-r, b_t-1])\\
&=(\b'' \sqcup [b_t-r-1, b_t-2], (b_t-r-1)^\vee)\in SpT_{2n}(\varpi_{t-1})\times SST_{2n}(\varpi_{2n-1}). 
\end{align*}

We have to show that
\begin{align}\bar \b:=\b'' \sqcup [b_t-r-1, b_t-2]=\b'\setminus\{b_t-1\}\sqcup \{b_t-r-1\}\in SpT_{2n}(\varpi_{t-1}).\label{proofsymp}
\end{align}

\begin{itemize}
\item For $r\notin 2\Z$ and $b_t\in 2\Z$, one has $b_t-r\notin 2\Z$, and $b_t-r-1\in 2\Z$.
\end{itemize}
We then write \eqref{proofsymp} as $$\eqref{proofsymp}\Leftrightarrow\b'' \sqcup \{b_t-r-1\} \sqcup[b_t-r, b_t-2].$$

If $b_t-r-1\in 2\Z$ is removable in \eqref{proofsymp}, by Proposition \ref{prop:rem}, $\s(b_t-r-1)=b_t-r-2\notin 2\Z$ is removable.
Since, from \eqref{ineq}, $b_{t-r-1}\le b_t-r-2$:  either $b_{t-r-1}= b_t-r-2$ which implies $b_{t-r-1}$ removable in $\b'':= (b_1, \dots , b_{t-r-1})\subseteq \b'\in SpT_{2n}(\varpi_{t-1})$, a contradiction; or $b_{t-r-1}< b_t-r-2$ and in this case $b_t-r-2\notin \b''$ and thus not removable from $\b'' \sqcup [b_t-r-1, b_t-2]$.

Let $r\ge 3$.  If $b_t-r$ is removable in \eqref{proofsymp}, by Proposition \ref{prop:rem},

$$b_t-r<2(t-r-1+2)\Leftrightarrow b_t<2t-r+2,~~~r\ge 3\Rightarrow b_t<2t-1.$$
A contradiction with $\b=(b_1,\dots,b_t)\in SpT_{2n}(\varpi_t)$.
 Therefore , $b_t-r, b_t-(r-1)$ are not removable in \eqref{proofsymp}.
 By reverse induction,  assume for $5\le j\le r$ that $b_t-j, b_t-(j-1)$ are not removable whenever $j\notin 2\Z$.
 Let us prove the claim for $j-2$ with $j\notin 2\Z$. By Proposition \ref{prop:rem}, suppose that we had
 \begin{align*}&b_t-(j-2)< 2(t-r-1+1+r-(j-2)+1)=2(t-(j-2)+1)\\
 &\Leftrightarrow b_t-(j-2)<2t-2(j-2)+2\Leftrightarrow b_t<2t-(j-2)+2\\
 &\Leftrightarrow b_t<2t-j+4\\
 &\Leftrightarrow b_t<2t-1, \mbox{ by $j\ge 5$},
 \end{align*}
 which is a contradiction with $\b=(b_1,\dots,b_t)$ symplectic.
 Hence, in this case, \eqref{proofsymp} holds.

 \medskip
\begin{itemize}
 \item For $2\le r\in 2\Z$ with $b_t\in 2\Z$, one has $b_t-r\in 2\Z$ and $b_t-r-1\notin 2\Z$.
\end{itemize}

 If $b_t-r-1\notin 2\Z$ is removable in \eqref{proofsymp}, by Proposition \ref{prop:rem}, $\s(b_t-r-1)=b_t-r\in 2\Z$ is removable in \eqref{proofsymp}. Therefore, by Proposition \ref{prop:rem}

 \begin{align*}
& b_t-r<2(t-r-1+1+1)-1\Leftrightarrow b_t-r<2t-2r+2-1\\
 &\Leftrightarrow b_t<2t-r+1\le 2t-1, \mbox{ by $r\ge 2$}\\
&\Leftrightarrow b_t<2t-1
\end{align*}
which is a contradiction with $\b=(b_1,\dots,b_t)$ symplectic. Hence $b_t-r-1, b_t-r$ are not removable in \eqref{proofsymp}.

By reverse induction,  assume for $4\le j\le r$ that $b_t-j-1, b_t-j$ are not removable whenever $j\in 2\Z$.
 Let us prove the claim for $j-2$ with $j\in 2\Z$. By Proposition \ref{prop:rem}, suppose that we had
 \begin{align*} &b_t-(j-2)< 2(t-r-1+r+1-(j-2)+1)-1=2(t-(j-2)+1)-1\\
 &\Leftrightarrow b_t-(j-2)<2t-2(j-2)+1\Leftrightarrow b_t<2t-(j-2)+1\\
 &\Leftrightarrow b_t<2t-j+3\\
 &\Leftrightarrow b_t<2t-1, \mbox{ by $j\ge 4$},
 \end{align*}
 which is a contradiction with $\b=(b_1,\dots,b_t)$ symplectic. Hence, also in this case, \eqref{proofsymp} holds.

Since $\bar\b=\b'\setminus\{b_t-1\}\sqcup \{b_t-r-1\}=(b_1, \dots , b_{t-r-1}) \sqcup [b_t-r-1, b_t-2]\in SpT_{2n}(\varpi_{t-1})$, by induction on $l$,
\begin{align}& \redu^{-1}_{l-1}(\bar \b)= \redu^{-1}_{l-1}((b_1, \dots , b_{t-r-1}) \sqcup [b_t-r-1, b_t-2])\nonumber\\
&=\bar\a\in SST_{2n}(\varpi_{l-1}), ~ \mbox{ and } \bar\b=(b_1, \dots , b_{t-r-1}) \sqcup [b_t-r-1, b_t-2]\subseteq \bar\a \label{hat1}.
\end{align}

By induction on $t\ge 0$, recalling \eqref{ineq}, 
define the column 
$$ \x:=\overline\a\cap [1,b_t-r-2]\supseteqq \b''=(b_1, \dots , b_{t-r-1}),$$ and   one has
\begin{align}\bar\a=\x\sqcup [b_t-r-1, b_t-2+s], \mbox{ for some $0\le s\in 2\Z$ }.
\label{hat2}
\end{align}
such that $\ell(\overline\a\cap [1,b_t-r-2])+r+s=l-1$.

Therefore
\begin{align*}&(K\otimes id)R_{l-1,2n-1}(\redu^{-1}_{l-1}\otimes id)(\overline\b, (b_t-r-1)^\vee)\\
&=(K\otimes id)R_{l-1,2n-1}(\x\sqcup [b_t-r-1, b_t-2+s], (b_t-r-1)^\vee),
\mbox{ by  \eqref{hat1}, \eqref{hat2}}\\
&(K\otimes id)(( b_t-2+s+1)^\vee,\x\sqcup [b_t-r, b_t-2+s+1]),
\mbox{ by \eqref{camel2}}\\
&=(K\otimes id)(\s( b_t-2+s+2)^\vee,\x\sqcup [b_t-r, b_t-2+s+1]),\\
&\mbox{ by $b_t-2+s+2\in 2\Z\Rightarrow \s(b_t-2+s+2)=b_t-2+s+1$,}\\
&=( b_t-2+s+2,\x\sqcup [b_t-r, b_t-2+s+1])\\
&\longrightarrow (\x\sqcup [b_t-r, b_t-2+s+1], b_t-2+s+2)=\x\sqcup [b_t-r, b_t-2+s+2]\in SST_{2n}(\varpi_{l}).
\end{align*}

Hence, for $b_t\in 2\Z$ and  $\b = (b_1, \dots , b_{t-r-1}) \sqcup [b_t-r, b_t]$ for some $r\ge 1$, with $b_{t-r-1}<b_t-r-1$,

\begin{align}&\redu^{-1}(\b)=\a=\overline\a\cap [1,b_t-r-2]\sqcup [b_t-r, b_t-2+s+2]\supseteqq \b= (b_1, \dots , b_{t-r-1}) \sqcup [b_t-r, b_t], \nonumber\\
&\mbox{ and }\nonumber\\
&\a=\overline\a\setminus \{b_t-r-1\}\sqcup [b_t-2+s+1, b_t-2+s+2], \mbox{ with $0\le s\in 2\Z$  in \eqref{hat2}}\nonumber\\
&\mbox{ where $\overline\a=\redu^{-1}((b_1, \dots , b_{t-r-1}) \sqcup [b_t-r-1, b_t-2]).$}\label{main5}
\end{align}
\end{enumerate}
\end{proof}

As a consequence of the previous theorem, the injectivity of the reduction map, and  Proposition 4.4.1 in \cite{watanabe}, the inverse reduction map $\redu^{-1}$ for $t\ge 2$ also decomposes, in reverse order,  into the inverses of the maps in \eqref{decomposeredu}
 (on suitable domains). For those inverses on the suitable domains, we now explicitly write the inverses needed in next theorem.

 Let $l\ge t\ge 2$,

 \begin{align}\bigwedge&:\{((a),a_1,\dots,a_{t-1})\in SST_{2n}(\varpi_1)\times SST_{2n}(\varpi_{t-1}): a>a_{t-1}\}\longrightarrow SST_{2n}(\varpi_{t})\nonumber\\
  &~~~~~~~~~~~~~\qquad\qquad\qquad\qquad\qquad \qquad\qquad\qquad((a),(a_1,\dots, a_{t-1}))\mapsto(a_1,\dots, a_{t-1},a),
  \end{align}

  \begin{align}{\bigwedge}^{-1}&:SST_{2n}(\varpi_{t})\longrightarrow \{((a),a_1,\dots,a_{t-1})\in SST_{2n}(\varpi_1)\times SST_{2n}(\varpi_{t-1}): a>a_{t-1}\}\nonumber\\
  &~~~~~~~~~~~~~\quad (a_1,\dots,a_{t-1},a)\mapsto((a),(a_1,\dots, a_{t-1})),\label{restrictvee}
  \end{align}

  \begin{align}K&:SST_{2n}(\varpi_{1})\longrightarrow SST_{2n}(\varpi_{2n-1})\nonumber\\
  &~~~~~~~~~~~~~\quad\quad\quad\quad  (a)\mapsto\s(a)^\vee,
  \end{align}

 \begin{align}K^{-1}&:SST_{2n}(\varpi_{2n-1})\longrightarrow SST_{2n}(\varpi_1)\nonumber\\
  &~~~~~~~~~~~~~\quad\quad\qquad \quad \s(a)^\vee\mapsto(a), \mbox{ $a\in [2n]$},
  \end{align}

  \begin{align}{\bigvee}&:SST_{2n}(\varpi_l)\longrightarrow  \{((a), (a_1,\dots, a_{l-1}))\in SST_{2n}(\varpi_1)\times SST_{2n}(\varpi_{l-1}): a> a_{l-1}\}\nonumber\\
  &~~~~~~~~~~~~~\qquad\qquad \qquad\qquad(a_1,\dots,a_{l-1},a)\mapsto((a),(a_1,\dots, a_{l-1})),\label{restrictwedge}
  \end{align}

  \begin{align}{\bigvee}^{-1}&: \{((a), (a_1,\dots, a_{l-1}))\in SST_{2n}(\varpi_1)\times SST_{2n}(\varpi_{l-1}): a>a_{l-1}\}\longrightarrow SST_{2n}(\varpi_l)\nonumber\\
  &~~~~~~~~~~~~~\qquad\qquad \qquad\qquad((a),(a_1,\dots,a_{l-1}))\mapsto(a_1,\dots, a_{l-1},a),\label{restrictwedgeinverse}
  \end{align}

\begin{thm}\label{inversemaps}
For a fixed $n\in\mathbb{N}$, let $l\in [0,2n]$ and $t\in[0,n]$ such that  $0\le t\le min\{l,2n-l\}$ and $l-t\in 2\Z$. Then
\begin{align}
&\redu^{-1}=\redu_t^{-1}:SpT_{2n}(\varpi_t)\rightarrow SST_{2n}(\varpi_l)\nonumber\\
\end{align}
is inductively defined as follows
\begin{enumerate}
\item for $t=0$, $l\in 2\Z$ in which case $\redu^{-1}(())=(1,\dots, l)$,
\item for $t=1$,  $l\notin 2\Z$ in which case for ${\bf b}=(b) \in SpT_{2n}(\varpi_1)$, one has
$$\redu^{-1}({\bf b})=(1,\dots,l_1)(b) T_1\in SST_{2n}(\varpi_l),$$ where

\begin{align}({b})T_1=
\begin{cases}
({b}) ( b+1,\dots, b+l_{2}),
  &
 \mbox{if } b\in 2\mathbb{Z}\\
({b}) ( b+1+1,b+1+2,\dots, b+1+l_{2}),
 &\mbox{if } b\notin 2\mathbb{Z}\\
\end{cases},
\end{align}
and
\begin{align}
&l_1=
\begin{cases} min\{b-2, l-1\},& \mbox{if } b\in 2\mathbb{Z}\\
min\{ b-1, l-1\},&\mbox{if } b\notin 2\mathbb{Z},\\
\end{cases}
\qquad \mbox{ and }  l_{2}=l-1-l_1.
\end{align}

\item for $t\ge  2$,   one has $l\ge t\ge 2$, and
\begin{align}
\redu^{-1}={\bigvee}^{-1}\circ (K^{-1},id)\circ R_{l-1,2n-1}\circ (\redu_{t-1}^{-1},id)\circ R_{2n-1,t-1}\circ (K,id)\circ {\bigwedge}^{-1}. \label{inversefactorization}
\end{align}
\end{enumerate}

\end{thm}

\begin{proof} The composition \eqref{inversefactorization} is well defined by Proposition 4.4.1 in \cite{watanabe} and because $(1,2)\notin SpT_{2n}(\varpi_2)$. Also in this case $\pi=id$ and ${\bigvee}^{-1}, {\bigwedge}^{-1}$,  restrict to suitable domains as in \eqref{restrictvee} and
\eqref{restrictwedge}.

Assertion $(1)$  follows from Corollary \ref{empty0}. Assertion $(2)$ follows from Proposition \ref{prop:rem}. The computations show $\rem((1,\dots,l_1)(b) T_1)=(1,\dots,l_1)T_1$ and $\redu((1,\dots,l_1)(b)T_1)=(b)$.

We prove assertion $3$. Consider the maps above and Proposition 4.4.1 in \cite{watanabe}. Let $l\ge t\ge 2$ and ${\bf b}=(b_1,\dots,b_t)\in SpT_{2n}(\varpi_t)$. Then $\redu^{-1}(\bf b)$ is obtained as follows:

\begin{align}&{\bf b}=(b_1,\dots, b_t)\in SpT_{2n}(\varpi_t)\overset{{\bigwedge}^{-1}}\longrightarrow ((b_t),(b_1,\dots,b_{t-1}))\in SpT_{2n}(\varpi_1)\times SpT_{2n}(\varpi_{t-1})\nonumber\\
&\overset{K\otimes id}\longrightarrow (\s(b_t)^\vee, {\bf b'})\in SST_{2n-1}(\varpi_{2n-1})\times SpT_{2n}(\varpi_{t-1}),~~{\bf b'}~=(b_1,\dots,b_{t-1})\nonumber\\
&\overset{R_{2n-1,t-1}}\longrightarrow ({\bf c'},{\bf d}), \mbox{ for some $(\c',\d) \in SpT_{2n}(\varpi_{t-1})\times SST_{2n}(\varpi_{2n-1})$}\nonumber\\
&\overset{\redu_{l-1}^{-1}\otimes id}\longrightarrow ({\bf c}, {\bf d})\in SST_{2n}(\varpi_{l-1})\times SST_{2n}(\varpi_{2n-1})\nonumber\\
&\overset{R_{l-1, 2n-1}}\longrightarrow (\s(a_l)^\vee,\a'= (a_1,\dots,a_{l-1}))\in SST_{2n}(\varpi_{2n-1})\times SST_{2n}(\varpi_{l-1}) \mbox{ for some $a_l>a_{l-1}$}\nonumber\\
&\overset{K^{-1}\otimes id}\longrightarrow ((a_l),\a'=(a_1,\dots,a_{l-1})) \in SST_{2n}(\varpi_{1})\times SST_{2n}(\varpi_{l-1}), \nonumber\\
&\overset{{\bigvee}^{-1}}\longrightarrow \a=(a_1,\dots,a_l)\in SST_{2n}(\varpi_l).
\end{align}
\end{proof}
\subsection{Main Theorem}
As a bonus of the lengthy proof of Proposition \ref{prop:surjective} we get the Main Theorem below which packs the information in \eqref{inducxx}, \eqref{main1}, \eqref{rankrelation1}, \eqref{main2}, \eqref{main3}, \eqref{main4},and \eqref{main5}.

\begin{thm}(Main Theorem) \label{th:main} For a fixed $n\in\mathbb{N}$, let $l\in[0,2n]$. The inverse reduction map

\begin{align}\redu^{-1}=\redu^{-1}_l:\bigsqcup_{\begin{smallmatrix} 0\le t\le min\{l,2n-l\}\\
l-t\in 2\mathbb{Z}
\end{smallmatrix}}SpT_{2n}(\varpi_t)&\rightarrow SST_{2n}(\varpi_l), \quad \b=(b_1,\dots,b_{t-1},b_t)\mapsto \redu^{-1}(\b)=\a
\end{align}
 is such that
 \begin{enumerate}
 \item for $l=0,1$, either $t=0$ or respectively $t=1$, in which cases $\redu^{-1}=id$,
  where $SpT_{2n}(\varpi_0)=SST_{2n}(\varpi_0)=\{()\}$ and $SpT_{2n}(\varpi_1)=SST_{2n}(\varpi_1)$.

  \item for $l=2$ either $t=0$ in which case ${\bf b}=()$   and $\redu^{-1}()=\a=(1,2)\in SST_{2n}(\varpi_2)$; or respectively $t=2$  in which case $\b=(b_1,b_2) \in SpT_{2n}(\varpi_2)=SST_{2n}(\varpi_2)\setminus \{(1,2)\}$,  and $\redu^{-1}(\b)=\b$.
 \item for $l\ge 3$, if $t=0$ and $l\in 2\Z$, then ${\bf b}=()$   and $\a=(1,2,\dots,l)\in SST_{2n}(\varpi_l)$; and if $t=1$ and $3\le l\notin 2\Z$, then $\b=(b)\in \SpT(\varpi_1)$, and
     $\redu_l^{-1}((b))=\a\in\SST(\varpi_l)$ defined by
 \begin{align}\label{induczeronethm}\hat\a:=[1,l-1]=\redu_{l-1}^{-1}()
\end{align}
and
\begin{align}\label{inducxxthm}\a=\begin{cases}(\hat \a,b),& \mbox{ if $\hat \a\subseteq[1,b-1]$},\\
\hat \a\setminus \{b+1\}\sqcup [l,l+1],& \mbox{ if $b\notin 2\Z$ and  $b\in\hat \a=[1,l-1]$},\\
\hat \a\setminus \{b-1\}\sqcup  [l,l+1],& \mbox{ if $b\in 2\Z$ and  $b\in\hat \a=[1,l-1]$},\\
\end{cases}
\end{align}

 \item Let $l\ge 3$ and  $t\ge 2$.
 \begin{itemize}
 \item for $\b\in SpT_{2n}(\varpi_t)$ without a nonempty interval fixed by the parity involution $\s$ as a suffix, define $$\hat\a:=\redu_{l-1}^{-1}(b_1,\dots,b_{t-1})\in SST_{2n}(\varpi_{l-1}).$$ Then

     \begin{align}\label{inducthm}\hat\a=\begin{cases}\hat\a\cap [1,b_{t-1}-1]\sqcup \{b_{t-1}\}\sqcup [b_{t-1}+1,b_{t-1}+s],~~\mbox{ }& \mbox{ if $b_{t-1}\in 2\Z$}\\
\hat\a\cap [1,b_{t-1}-1]\sqcup \{b_{t-1}\}\sqcup [b_{t-1}+1+1,b_{t-1}+1+s],~~\mbox{ }& \mbox{ if $b_{t-1}\notin 2\Z$},
\end{cases}
\end{align}
where  $0\le s\in 2\Z$ is  such that   $s+\ell(\hat\a\cap[1,b_{t-1}])=l-1$
and

\begin{align}\label{rankmain}\a=\begin{cases}(\hat \a,b_t),& \mbox{ if $\hat \a\subseteq[1,b_t-1]$ and $b_t\notin 2\Z$},\\
\hat\a\setminus\{\s(b_t)\}\sqcup [b_{t-1}+1+s, b_{t-1}+2+s],& \mbox{ if $b_t\in\hat \a$ , $b_{t-1}\in 2\Z$, $b_t\notin 2\Z$},\\
\hat\a\setminus\{\s(b_t)\}\sqcup [b_{t-1}+2+s, b_{t-1}+3+s],& \mbox{ if $b_t\in\hat \a$ , $b_{t-1}\notin 2\Z$, $b_{t}\notin 2\Z$},\\
(\hat \a,b_t),& \mbox{ if $\hat \a\subseteq[1,b_t-2]$ and $b_{t-1}<b_t-1$, $b_t\in 2\Z$},\\
\hat\a\setminus\{\s(b_t)\}\sqcup [b_{t-1}+s+1,b_{t-1}+s+2], & \mbox{ if $\s(b_t)\in\hat \a$ and $b_{t-1}<b_t-1$, $b_{t-1}\in 2\Z$, $b_t\in 2\Z$},\\
\hat\a\setminus\{\s(b_t)\}\sqcup [b_{t-1}+s+2,b_{t-1}+s+3], & \mbox{ if $\s(b_t)\in\hat \a$ and $b_{t-1}<b_t-1$, $b_{t-1}\notin 2\Z$, $b_t\in 2\Z$}.\\
\end{cases}
\end{align}
where $0\le s\in 2\Z$ and $s+\ell(\hat\a\cap[1,b_{t-1}])=l-1.$
\item otherwise,   $\b = (b_1, \dots , b_{t-r-1}) \sqcup [b_t-r, b_t]\in  SpT_{2n}(\varpi_t)$, for some $1\le r<t$, with $b_t\in 2\Z$,  $b_{t-r-1}<b_t-r-1<b_{t-r}=b_t-r$,  and define
    $$\b\setminus \{b_t,\s(b_t)\}\sqcup\{b_t-r-1\}=(b_1, \dots , b_{t-r-1}) \sqcup [b_t-r-1, b_t-2]\in \SpT(\varpi_{t-1}), \mbox{ and }$$
$$\overline\a:=\redu_{l-1}^{-1}((b_1, \dots , b_{t-r-1}) \sqcup [b_t-r-1, b_t-2])\in \SST(\varpi_{l-1}).$$
Then

\begin{align}\bar\a=\bar\a\cap [1,b_t-r-2]\sqcup [b_t-r-1, b_t-2+s], \mbox{ for some $0\le s\in 2\Z$ },
\label{hat2thm}
\end{align}
such that $\ell(\overline\a\cap [1,b_t-r-2])+r+s=l-1$, and

\begin{align}\a=\overline\a\setminus \{b_t-r-1\}\sqcup [b_t-2+s+1, b_t-2+s+2], \mbox{ for some $0\le s\in 2\Z$, }
\end{align}
such that $\ell(\overline\a\cap [1,b_t-r-2])+r+s=l-1$.
\end{itemize}

\end{enumerate}
\end{thm}

\begin{obs} To compute $\hat\a$ or $\overline\a$ we  may get assisted by the explicit theorems in \cite{azreduction}.
\end{obs}
\subsection{Illustration of the Main Theorem \ref{th:main}}

In  Example 4, $(3)$ in \cite{azreduction}, for $n=9$, $l=10$,  and $t=6<min\{l=10,18-10\} $ , we have considered ${\bf b} =1,7,(11,12,13,14)\in SpT_{2n}(\varpi_6)$,  and
\begin{align}\redu_t^{-1}({\bf b})={\bf 1}(3,4){\bf 7}({\bf 11,12,13,14}) (15,16)\in SST_{18}(\varpi_{10}),\label{bees}\end{align}
 is easily computed  using our explicit formulas in \cite[Theorem 3]{azreduction}. We shall now compare this global method  by considering the inverse reduction on a  lower rank symplectic column as computed next,
$$1,7,10,(11,12)\in SpT_{2n}(\varpi_9).$$
Note this symplectic column is not a sub-column of $\b$.
We first compute  $\redu^{-1}({\bf b})$ using $\redu^{-1}$ as a composition of several maps \eqref{inversefactorization} and will see the usefulness of explicit formulas when large symplectic columns are considered. Then we avoid the composition of several maps and just use Theorem  \ref{th:main} where the computation   remains essentially on the computation of the inverse reduction map on a symplectic column of rank reduced by one, in this case \eqref{bees}.

\begin{ex} With the previous setup, Example 4, $(3)$ in \cite{azreduction}, $n=9$,  $t=6$, and $l=10$,
\begin{itemize}
\item using Theorem \ref{inversemaps},
\end{itemize}
\begin{align}&{\bf b}=(1,7,11,12,13,14)\in SpT_{2n}(\varpi_t)\overset{\pi^{-1}=id}\longrightarrow {\bf a}\in SpT_{2n}(\varpi_t)\nonumber\\
&\overset{\bigwedge^{-1}}\longrightarrow ((14),(1,7,11,12,13))\in  SpT_{2n}(\varpi_{1})\times SpT_{2n}(\varpi_{t-1})\nonumber\\
&\overset{K\otimes id}\longrightarrow (\s(14)^\vee;(1,7,11,12,13))\in SST_{2n}(\varpi_{2n-1})\times SpT_{2n}(\varpi_{t-1})\nonumber\\
&\overset{R_{2n-1,t-1}}\longrightarrow(({ 1,7,10,11,12});10^\vee)\in  SpT_{2n}(\varpi_{t-1})\times SST_{2n}(\varpi_{2n-1}),  \mbox{ by $[1,9] \sqcup[11,18]=10^\vee$},\nonumber\\
&\overset{\redu^{-1}_{t-1}\otimes id}\longrightarrow (({ 1,3,4,7})\sqcup [{ 10,14}];10^\vee)\in  SST_{2n}(\varpi_{l-1})\times SST_{2n}(\varpi_{2n-1}),\label{eq:unavoidable}
\end{align}
\begin{align}
&
\mbox{ by  \cite[Theorem 3]{azreduction}, $\redu^{-1}_{t-1}(1,7,10,11,12)=1,(3,4)7,10,11,12,(13,14)=(1,3,4,7)\sqcup [10,14]$,}\label{useful}\\
&\overset{R_{l-1,2n-1}}\longrightarrow (\s^\vee(16); (1,(3,4),7))\sqcup [11,15])\in   SST_{2n}(\varpi_{2n-1})\times SST_{2n}(\varpi_{l-1}), \mbox{ by Proposition \ref{prop:collect}}, \eqref{camel},\nonumber\\
&\overset{K^{-1}\otimes id}\longrightarrow (16; 1,(3,4),7,(11,12,13,14),15)\in   SST_{2n}(\varpi_{1})\times SST_{2n}(\varpi_{l-1}),\nonumber\\
&\mbox{ by $(1,(3,4),7,(11,12,13,14),15))=(1,(3,4),7))\sqcup [11,15]$,}\nonumber\\
&\overset{\bigvee^{-1}}\longrightarrow ( 1,(3,4),7,(11,12,13,14),15,16)=\redu^{-1}({\bf b})\in SST_{2n}(\varpi_{l}),\label{solution}
\end{align}
\begin{itemize}
\item We now check with Theorem \ref{th:main}, \eqref{hat2thm}. One has $l=10$,  $r=3$, $t=6$, $t-r-1=6-3-1=2$ and $\b=(1,7)\sqcup[11,14]$ with $b_6=14\in 2\Z$, $b_2=7<b_6-2-1=11<b_6-3=11=b_3$. One has $b_t-r-1=10$ and $b_t-r-2=9$.

Define 
    $$\b\setminus \{b_t,\s(b_t)\}\sqcup\{b_t-r-1\}=(1,7) \sqcup [10, 12]=(1,7,10)\sqcup[11,12]\in \SpT(\varpi_{t-1}), \mbox{ and }$$

$$\overline\a:=\redu_{l-1}^{-1}((b_1, \dots , b_{t-r-1}) \sqcup [b_t-r-1, b_t-2])=\redu_{l-1}^{-1}((1,7,10)\sqcup[11,12])$$
$$=(1,3,4,7)\sqcup [10,14]\in \SST(\varpi_{l-1}) \mbox{ by  \cite[Theorem 3]{azreduction}}.$$

Then \begin{align}&\redu_l^{-1}(\b)=\a=\overline\a\setminus \{b_t-r-1\}\sqcup [b_t-2+s+1, b_t-2+s+2]\nonumber\\
&=(1,3,4,7)\sqcup [11,14] \sqcup[15,16]=\eqref{solution}
\end{align}
such that $0\le s\in 2\Z$ and 
$$\ell(\overline\a\cap [1,b_t-r-2])+r+s=l-1\Leftrightarrow\ell(\overline\a\cap [1,9])+3+s=9\Leftrightarrow 4+3+s=9$$
$$\Leftrightarrow s=2.$$
\end{itemize}
\end{ex}


\bibliography{sample17}
\bibliographystyle{alpha}

\end{document}